\documentclass[12pt]{amsart}

\usepackage{amsthm, amsmath, amsfonts, amssymb, graphicx, tikz-cd, comment, array, hyperref, extarrows, mathtools, todonotes}
\usepackage{enumitem}
\usepackage[capitalise]{cleveref} 
\DeclareMathAlphabet{\mathcal}{OMS}{zplm}{m}{n}

\newtheorem{theorem}{Theorem}[section]
\newtheorem{lemma}[theorem]{Lemma}
\newtheorem{proposition}[theorem]{Proposition}
\newtheorem{corollary}[theorem]{Corollary}

\newtheorem{claim}[theorem]{Claim}
\theoremstyle{definition}
\newtheorem{definition}[theorem]{Definition}

\newtheorem{example}[theorem]{Example}

\newtheorem{remark}[theorem]{Remark}

\def\KHaus{{\sf{KHaus}}}
\def\Stone{{\sf{Stone}}}
\newcommand{\StoneR}{\mathsf{Stone}^{\mathsf{R}}}

\newcommand{\KHausR}{\mathsf{KHaus}^\mathsf{R}}

\def\ba{{\sf{BA}}}
\def\ma{{\sf{MA}}}

\def\dev{{\sf DeV}}

\def\KRFrm{{\sf KRFrm}}

\def\KHausR{{\sf{KHaus}^R}}
\def\StoneR{{\sf{Stone}^R}}
\def\KRFrm{{\mathsf{KRFrm}}}

\newcommand{\BAS}{\mathsf{BA}^\mathsf{S}}
\newcommand{\KRFrmP}{{\mathsf{KRFrm^P}}}

\newcommand{\StoneER}{\mathsf{StoneE}^\mathsf{R}}

\newcommand{\SubSfive}{\mathsf{SubS5^S}}

\newcommand{\devS}{\mathsf{DeV^S}}

\def\Clop{{\sf{Clop}}}
\def\Uf{{\sf{Uf}}}
\def\RO{{\mathcal{RO}}}

\def\V{{\mathbb{V}}}
\def\Q{{\mathcal{Q}}}
\def\G{{\mathcal{G}}}
\def\int{{\sf{int}}}
\def\cl{{\sf{cl}}}
\def\K{{\mathbb{K}}}
\def\L{{\mathbb{L}}}
\def\I{{\mathbb{I}}}
\def\J{{\mathbb{J}}}
\def\M{{\mathbb{M}}}

\def\O{{\mathcal{O}}}
\def\Alg{{\sf{Alg}}}
\def\Coalg{{\sf{Coalg}}}

\def\D{\mathcal{D}}

\def\S{\mathrel{S}}
\def\E{\mathrel{E}}
\def\R{\mathrel{R}}
\def\T{\mathrel{T}}
\newcommand{\id}{\mathrm{id}}
\newcommand{\Rbox}{\mathcal{R}_\Box}
\newcommand{\Rdia}{\mathcal{R}_\Diamond}
\newcommand{\Rem}{\mathcal{R}_{EM}}
\newcommand{\Rboxrel}{\mathrel{\Rbox}}
\newcommand{\Rdiarel}{\mathrel{\Rdia}}

\newcommand{\Vr}{\V^{\mathsf{R}}}
\newcommand{\Vrbox}{\V^{\mathsf{R}}_\Box}
\newcommand{\Vrdia}{\V^{\mathsf{R}}_\Diamond}
\newcommand{\Sbox}{\mathcal{S}_\Box}
\newcommand{\Sdia}{\mathcal{S}_\Diamond}
\newcommand{\Sem}{\mathcal{S}}
\newcommand{\Sboxrel}{\mathrel{\Sbox}}
\newcommand{\Sdiarel}{\mathrel{\Sdia}}
\newcommand{\Semrel}{\mathrel{\Sem}}
\newcommand{\Ks}{\K^{\mathsf{S}}}
\newcommand{\Jp}{\J^{\mathsf{P}}}
\newcommand{\Ksbox}{\K^{\mathsf{S}}_\Box}
\newcommand{\Ksdia}{\K^{\mathsf{S}}_\Diamond}
\newcommand{\Ls}{\L^\mathsf{S}}

\title[Vietoris endofunctor for closed relations and its de Vries dual]{Vietoris endofunctor for closed relations\\ and its de Vries dual}

\author[M. Abbadini]{Marco Abbadini}
\address{Dipartimento di Matematica,
Universit\`a degli studi di Salerno,
84084 Fisciano (SA),
Italy}
\email{mabbadini@unisa.it}

\author[G. Bezhanishvili]{Guram Bezhanishvili}
\address{Department of Mathematical Sciences\\
New Mexico State University\\
Las Cruces NM 88003\\
USA}
\email{guram@nmsu.edu}

\author[L. Carai]{Luca Carai}
\address{Departament de Filosofia,
Universitat de Barcelona,
08001 Barcelona,
Spain}
\email{luca.carai.uni@gmail.com}

\date{}



\begin{document}

\begin{abstract}
We generalize the classic Vietoris endofunctor to the category of compact Hausdorff spaces and closed relations. The lift of a closed relation is done by generalizing the construction of the Egli-Milner order.
We describe the dual endofunctor on the category of de Vries algebras and subordinations. This is done in several steps, by first generalizing the construction of Venema and Vosmaer to the category of boolean algebras and subordinations, then lifting it up to $\sf S5$-subordination algebras, and finally using MacNeille completions to further lift it to de Vries algebras. Among other things, this yields a generalization of Johnstone's pointfree construction of the Vietoris endofunctor to the category of compact regular frames and preframe homomorphisms.
\end{abstract}

\maketitle

\tableofcontents

\section{Introduction}

Hyperspace constructions are among classic constructions in topology \cite{Eng89} with numerous applications \cite{GHKLMS03}. One of the most famous such is taking the Vietoris space $\V(X)$ of a compact Hausdorff space $X$. It is well known \cite[Sec.~III.4]{Joh82} that this defines an endofunctor $\V$ on the category $\KHaus$ of compact Hausdorff spaces and continuous functions.

In recent years there has been some work on generalizing $\KHaus$ to the category $\KHausR$ of compact Hausdorff spaces and closed relations \cite{Tow96,JKM01,Moshier2004,BezhanishviliGabelaiaEtAl2019,ABC22a}. One advantage of $\KHausR$ is that it is more symmetric. In particular, $\KHausR$ is a dagger category (see, e.g., \cite[Rem.~3.7]{BezhanishviliGabelaiaEtAl2019}). It is then natural to try to generalize the Vietoris endofunctor to an endofunctor $\Vr \colon \KHausR \to \KHausR$. This is one of the aims of the present paper. 

If $f\colon X\to Y$ is a continuous function between compact Hausdorff spaces, then its lift $\V(f) \colon \V(X)\to\V(Y)$ is defined by $\V f(F)=f[F]$, where $F$ is a closed subset of $X$ and $f[F]$ is the $f$-image of $F$ in $Y$. We observe that if $R\subseteq X\times Y$ is a closed relation, then the lift $\Vr(R)$ can be defined by generalizing the well-known Egli-Milner order, which originates in domain theory (see, e.g., \cite[p.~224]{Smy91}). More precisely, if $F\subseteq X$ and $G\subseteq Y$ are closed sets, then we define $\Vr(R)$ by 
\[
F \ \Vr(R) \ G \ \Longleftrightarrow \ G \subseteq R[F] \mbox{ and } F \subseteq R^{-1}[G],
\] 
where $R[F]$ is the $R$-image of $F$ in $Y$ and $R^{-1}[G]$ is the $R$-inverse image of $G$ in $X$. We show that this defines an endofunctor $\Vr \colon \KHausR \to \KHausR$, which restricts to the Vietoris endofunctor $\V \colon \KHaus \to \KHaus$ (see \cref{sec: V on KHausR}).

In pointfree topology there is a well-known duality between $\KHaus$ and the category $\KRFrm$ of compact regular frames \cite{Isb72} (see also \cite{BM80} and \cite[Sec.~III.1]{Joh82}), known as Isbell duality. It is obtained by associating with each compact Hausdorff space $X$ the frame $\mathcal{O}(X)$ of opens of $X$. In \cite[Sec.~III.4]{Joh82} (see also \cite{Joh85}) Johnstone described an endofunctor $\J \colon \KRFrm \to \KRFrm$ dual to $\V \colon \KHaus \to \KHaus$.

Isbell duality extends to a duality between $\KHausR$ and the category $\KRFrmP$ of compact regular frames and preframe homomorphisms \cite{Tow96,JKM01}. While it is possible to extend Johnstone's endofunctor to an endofunctor $\J^{\sf P}\colon \KRFrmP\to\KRFrmP$ that is dual to $\Vr$, the construction is less elegant. Indeed, for $L\in\KRFrm$, the frame $\J(L)$ is constructed as the quotient of the free frame over the set $\{ \Box_a,\Diamond_a \mid a \in L\}$ by appropriate relations \cite[p.~112]{Joh82}. 
Then $\J(L)$ is a compact regular frame and the functoriality of $\J$ follows since every frame homomorphism preserves arbitrary joins and finite meets~\cite[pp.~113--115]{Joh82}.
Since $\KRFrmP$ is a wide subcategory of $\KRFrm$, to define $\Jp(L)$ we don't need to change the construction of $\J(L)$, but the functoriality of $\Jp$ becomes a non-trivial issue because preframe homomorphisms don't preserve arbitrary joins (see \Cref{rem:Johnstone}).
We instead take a different route.

There is another duality for $\KHaus$ that is closely related to Isbell duality. It is obtained by associating with each compact Hausdorff space $X$ the boolean frame $\RO(X)$ of regular opens of $X$ equipped with the proximity relation given by $U \prec V$ iff $\cl(U)\subseteq V$. This yields a duality between $\KHaus$ and the category of pairs $(B,\prec)$ where $B$ is a boolean frame and $\prec$ is a proximity relation on $B$. These pairs are known as de Vries algebras, the resulting category is denoted by $\dev$, and the duality between $\KHaus$ and $\dev$ is known as de Vries duality \cite{DeV62} (see also \cite{Bez10}). 

Since both $\KRFrm$ and $\dev$ are dually equivalent to $\KHaus$, the two categories are equivalent, and the equivalence is obtained through the booleanization $\mathfrak B\colon\KRFrm\to\dev$ \cite{Bez12}. The endofunctor $\J\colon\KRFrm\to\KRFrm$ then gives rise to an endofunctor on $\dev$ that is dual to the Vietoris endofunctor on $\KHaus$. A direct pointfree construction of this endofunctor, without utilizing $\J$, remained an open problem \cite[p.~375]{BBH15b}. The approach we develop in this paper, among other things, will lead to a solution of this problem (see \Cref{rem:open problem}).

In \cite{ABC22a} we extended de Vries duality to $\KHausR$. For this we worked with subordination relations, which originate in \cite{BBSV17}. We introduced the category $\BAS$ whose objects are boolean algebras and whose morphisms are subordination relations between them (see \cref{{sec: subordinations}}). Let $\StoneR$ be the full subcategory of $\KHausR$ consisting of Stone spaces. Then Stone duality extends to an equivalence between $\BAS$ and $\StoneR$ \cite{Cel18,KurzMoshierEtAl2023,ABC22a}. Using the machinery of allegories \cite{FS90} yields an equivalence between $\KHausR$ and the category whose objects are pairs $(B,S)$ where $B$ is a boolean algebra and $S$ is a subordination relation on $B$ satisfying axioms that generalize the axioms of an $\sf S5$-modality, which plays a prominent role in modal logic \cite{BRV01}. Because of this connection, we termed such pairs $(B,S)$ $\sf S5$-subordination algebras and denoted the resulting category by $\SubSfive$ \cite[p.~8]{ABC22a}. The category $\devS$ is a full subcategory of $\SubSfive$ consisting of de Vries algebras. It turns out that $\devS$ is equivalent to $\KHausR$, and hence to $\SubSfive$ \cite[p.~12]{ABC22a}. In fact, the equivalence between $\SubSfive$ and $\devS$ is obtained by generalizing the well-known MacNeille construction to $\sf S5$-subordination algebras \cite{ABC22b}. 

In this paper we define an endofunctor on $\SubSfive$ which is dual to the Vietoris endofunctor $\V^{\sf R}$ on $\KHausR$. This we do as follows. In \cite{VV14} the authors define the endofunctor $\K$ on the category $\ba$ of boolean algebras that is dual to the Vietoris endofunctor $\V$ on the category $\Stone$ of Stone spaces. One of our main technical results lifts this endofunctor to the endofunctor $\K^{\sf S}$ on $\BAS$ (see \cref{{sec: K on BAS}}). In \cref{sec: connecting K and V} we show that $\K^{\sf S}$ is equivalent to $\V^{\sf R}$ on $\StoneR$. Finally, in \cref{sec: extending K to L} we lift $\K^{\sf S}$ to an endofunctor on $\SubSfive$ which is equivalent to $\V^{\sf R}$ on $\KHausR$. We then show that composing the MacNeille completion functor with this endofunctor yields an endofunctor $\Ls$ on $\devS$ which is also equivalent to $\V^{\sf R}$. 

This resolves the problem mentioned above in the more general setting of the category $\SubSfive$, and its full subcategory $\devS$. In \cite{ABC23} we show how this yields a solution of the problem for the category $\dev$ by utilizing the technique developed in \cite{ABC22a}.

\section{The Vietoris endofunctor on \texorpdfstring{$\KHausR$}{KHaus{\textasciicircum}R} } \label{sec: V on KHausR}

We start by recalling the classic Vietoris construction.

\begin{definition}
Let $X$ be a compact Hausdorff space and $\O(X)$ the frame of opens of $X$. The {\em Vietoris space} of $X$ is the space $\V(X)$ of closed subsets of $X$ topologized by the subbasis $\{\Box_U \mid U\in \O(X)\}\cup\{\Diamond_U \mid U\in \O(X)\}$, where
\[
\Box_U = \{ F \in \V(X) \mid F \subseteq U\} \mbox{ and } \Diamond_U=\{F\in\V(X) \mid F\cap U\ne\varnothing\}.
\]
\end{definition}

For a continuous function $f \colon X \to Y$ between compact Hausdorff spaces, define the function $\V(f) \colon \V(X) \to \V(Y)$ by $\V(f)[F]=f[F]$ for each $F \in \V(X)$. It is well known (see, e.g.,~\cite[p.~112]{Joh82}) that this defines an endofunctor on $\KHaus$ which we denote by $\V$.

The following useful lemma is well known (see \cite[p.~112]{Joh82}):

\begin{lemma}\label{l:Johnstone axioms}
Let $X\in\KHaus$ and $\mathcal{U} \subseteq \O(X)$. Then
\[
\begin{array}{ll}
\Box_{\bigcap \mathcal U}=\bigcap\{\Box_U \mid U \in \mathcal U \} & \mathcal U \mbox{ finite}, \\
\Box_{\bigcup \mathcal U}=\bigcup\{\Box_U \mid U \in \mathcal U \} & \mathcal U \mbox{ directed}, \\
\Diamond_{\bigcup \mathcal U}=\bigcup\{\Diamond_U \mid U \in \mathcal U \} & \mathcal U \mbox{ arbitrary}, \\
\Box_{U\cup V} \subseteq \Box_U \cup \Diamond_V, & \Box_U \cap \Diamond_V \subseteq \Diamond_{U\cap V}.
\end{array}
\]
\end{lemma}

As a simple consequence of \cref{l:Johnstone axioms} we obtain:

\begin{lemma}
$\{\Box_U\cap\Diamond_{V_1}\cap\cdots\cap\Diamond_{V_n} \mid U,V_1,\dots,V_n\in\O(X)$ and $V_1,\dots,V_n\subseteq U\}$ is a basis for $\V(X)$.
\end{lemma}

We recall that a relation $R \colon X \to Y$ between compact Hausdorff spaces is \emph{closed} if it is a closed subset of $X \times Y$. Let $\KHausR$ be the category of compact Hausdorff spaces and closed relations between them. Identity morphisms in $\KHausR$ are identity relations and composition is relation composition.
Our goal is to lift $\V$ to an endofunctor on $\KHausR$. 
For this we generalize the definition of the Egli-Milner order \cite[p.~224]{Smy91}.

\begin{definition}
Let $R \colon X \to Y$ be a closed relation between compact Hausdorff spaces. Define $\Rbox,\Rdia$, and $\Rem$ from $\V(X)$ to $\V(Y)$ as follows:
\begin{enumerate}
\item $F \Rboxrel G$ iff $G \subseteq R[F]$,
\item $F \Rdiarel G$ iff $F \subseteq R^{-1}[G]$,
\item $\Rem = \Rbox \cap \Rdia$.
\end{enumerate} 
\end{definition}

In order to show that these relations are closed, we utilize the following lemma (see, e.g., \cite[Lem.~2.12]{BBSV17}).

\begin{lemma}\label{lem:R closed equiv condition}
For a relation $R \colon X \to Y$ between compact Hausdorff spaces, the following are equivalent:
\begin{enumerate}[label=\normalfont(\arabic*), ref = \arabic*]
\item\label{lem:R closed equiv condition:item1} $R$ is closed,
\item\label{lem:R closed equiv condition:item2} $R[F]$ is closed for each closed $F \subseteq X$ and $R^{-1}[G]$ is closed for each closed $G \subseteq Y$,
\item\label{lem:R closed equiv condition:item3} $(x,y) \notin R$ implies the existence of two open neighborhoods $U_x$ and $U_y$ such that $R[U_x] \cap U_y = \varnothing$.
\end{enumerate}
\end{lemma}

To simplify notation (see \cref{lem:R R1 R2 closed,prop:description-on-morphisms}), for $E \subseteq Z$ we write $-E$ to denote the complement of $E$ in $Z$.

\begin{proposition}\label{lem:R R1 R2 closed}
$\Rbox$, $\Rdia$, and $\Rem$ are closed relations.
\end{proposition}

\begin{proof}
To see that $\Rbox$ is a closed relation, it is enough to show that \cref{lem:R closed equiv condition}\eqref{lem:R closed equiv condition:item3} is satisfied. Let $(F,G) \notin \Rbox$. Then $G\not\subseteq R[F]$. Therefore, there is $x\in G$ such that $x\notin R[F]$. Since $R[F]$ is closed, there are disjoint open sets $U,V$ such that $x\in U$ and $R[F]\subseteq V$. Thus, $F\subseteq -R^{-1}-V$, so $F\in\Box_{-R^{-1}-V}$. Also, $G\cap U\ne\varnothing$, so $G\in\Diamond_U$. We show that $\Rbox[\Box_{-R^{-1}-V}]\cap\Diamond_U=\varnothing$. If $K\in \Rbox[\Box_{-R^{-1}-V}]\cap\Diamond_U$, then $K\cap U\ne\varnothing$ and there is $H\in\Box_{-R^{-1}-V}$ such that $H \Rboxrel K$. Therefore, $H\subseteq -R^{-1}-V$ and $K\subseteq R[H]$. From $H\subseteq -R^{-1}-V$ it follows that $R[H]\subseteq V$, so $K\subseteq V$, which together with $K\cap U\ne\varnothing$ implies that $U,V$ aren't disjoint, a contradiction. Thus, $\Rbox$ is closed.

To see that $\Rdia$ is a closed relation, we again show that \cref{lem:R closed equiv condition}\eqref{lem:R closed equiv condition:item3} is satisfied. Let $(F,G) \notin \Rdia$. Then $F \not\subseteq R^{-1}[G]$. Therefore, there is $x\in F$ such that $x \notin R^{-1}[G]$. Since $R^{-1}[G]$ is closed, there are disjoint open $U,V$ such that $x\in U$ and $R^{-1}[G]\subseteq V$. Thus, $F\cap U\ne\varnothing$ and $G\subseteq -R-V$. This implies that $F\in\Diamond_U$ and $G \in \Box_{-R-V}$. We show that $\Rdia[\Diamond_U] \cap \Box_{-R-V} = \varnothing$. If $K\in \Rdia[\Diamond_U] \cap \Box_{-R-V}$, then $K\subseteq -R-V$ and there is $H \in \Diamond_U$ such that $H \Rdiarel K$. Therefore, $R^{-1}[K]\subseteq V$, $H\cap U\ne\varnothing$, and $H\subseteq R^{-1}[K]$. Thus, $H\subseteq V$ and $H\cap U\ne\varnothing$, which contradicts that $U,V$ are disjoint. Consequently, $\Rdia$ is closed.

Finally, since both $\Rbox$ and $\Rdia$ are closed relations, so is their intersection, hence $\Rem$ is a closed relation.
\end{proof}

\begin{definition}
For a morphism $R \colon X \to Y$ in $\KHausR$ define 
\[
\Vrbox(R) = \Rbox, \ \Vrdia(R) = \Rdia, \mbox{ and } \Vr(R) = \Rem.
\] 
\end{definition}

We next show that $\Vrbox, \Vrdia$ are semi-functors, while $\Vr$ is an endofunctor on $\KHausR$. We recall that, given two categories $\mathsf{C}$ and $\mathsf{D}$, a \emph{semi-functor} $F \colon \mathsf{C} \to \mathsf{D}$ is an assignment that maps objects of $\mathsf{C}$ to objects of $\mathsf{D}$, morphisms of $\mathsf{C}$ to morphisms of $\mathsf{D}$, and preserves composition (see, e.g., \cite{Hay85,Hoofman1993}).
The notion of a semi-functor is weaker than that of a functor since identity morphisms are not required to be preserved.

\begin{theorem}\label{t:Vr functor}
\hfill\begin{enumerate}[label=\normalfont(\arabic*), ref = \arabic*]
\item\label{t:Vr functor:item1} $\Vrbox, \Vrdia \colon \KHausR \to \KHausR$ are semi-functors.
\item\label{t:Vr functor:item2} $\Vr$ is an endofunctor on $\KHausR$ which restricts to the Vietoris endofunctor $\V$ on $\KHaus$.
\item\label{t:Vr functor:item3} Each of the three restricts to $\StoneR$. 
\end{enumerate}
\end{theorem}

\begin{proof}
\eqref{t:Vr functor:item1}. That $\Vrbox, \Vrdia$ are well defined follows from \cref{lem:R R1 R2 closed}. Let $R_1 \colon X_1 \to X_2$, $R_2 \colon X_2 \to X_3$ be morphisms in $\KHausR$ and $F\subseteq X_1$, $G\subseteq X_2$, $H\subseteq X_3$ closed subsets.  
If $F \mathrel{\Vrbox(R_1)} G$ and $G \mathrel{\Vrbox(R_2)} H$, then 
$G \subseteq R_1[F]$ and $H \subseteq R_2[G]$. Therefore, 
\[
H \subseteq R_2[G] \subseteq R_2[R_1[F]] = (R_2 \circ R_1)[F].
\] 
Thus, $F \mathrel{\Vrbox(R_2 \circ R_1)} H$. Conversely, if $F \mathrel{\Vrbox(R_2 \circ R_1)} H$, then $H \subseteq R_2[R_1[F]]$. Let $G=R_1[F]$. Then $G$ is a closed subset of $X_2$ such that $G 
\subseteq R_1[F]$ and $H \subseteq 
R_2[G]$. Therefore, $F \mathrel{\Vrbox(R_1)} G$ and $G \mathrel{\Vrbox(R_2)} H$. Thus, $\Vrbox(R_2 \circ R_1) = \Vrbox(R_2) \circ \Vrbox(R_1)$. The proof for $\Vrdia$ is similar.

\eqref{t:Vr functor:item2}. That $\Vr$ is well defined follows from \cref{lem:R R1 R2 closed}. It is an immediate consequence of the definition of $\Rem$ that $\Vr$ preserves identity relations. If $F \mathrel{\Vr(R_1)} G$ and $G \mathrel{\Vr(R_2)} H$, then 
\[
G \subseteq R_1[F], \ F \subseteq R_1^{-1}[G], \ H \subseteq R_2[G], \ \text{and} \ G \subseteq R_2^{-1}[H].
\]
The same argument as in \eqref{t:Vr functor:item1} yields that $F \mathrel{\Vr(R_2 \circ R_1)} H$.
Conversely, let $F\subseteq X_1$ and $H\subseteq X_3$ be closed subsets  
such that $F \mathrel{\Vr(R_2 \circ R_1)} H$. Then $H \subseteq R_2[R_1[F]]$ and $F \subseteq R_1^{-1}[R_2^{-1}[H]]$. Since $R_1$ and $R_2$ are closed relations, $G \coloneqq R_1[F] \cap R_2^{-1}[H]$ is a closed subset of $X_2$. We show that $F \mathrel{\Vr(R_1)} G$ and $G \mathrel{\Vr(R_2)} H$. The definition of $G$ yields that $G \subseteq R_1[F]$ and $G \subseteq R_2^{-1}[H]$. Since $H \subseteq R_2[R_1[F]]$, for each $z \in H$ there are $x \in F$ and $y \in X_2$ such that $x \R_1 y$ and $y \R_2 z$. Therefore, $y \in R_1[F] \cap R_2^{-1}[H] = G$. Thus, $H \subseteq R_2[G]$. Similarly, $F \subseteq R_1^{-1}[R_2^{-1}[H]]$ implies $F \subseteq R_1^{-1}[G]$. Consequently, $F \mathrel{\Vr(R_1)} G$ and $G \mathrel{\Vr(R_2)} H$. This shows that $\Vr$ preserves compositions, and hence is an endofunctor on $\KHausR$.

If $f \colon X_1 \to X_2$ is a continuous function, then $F \subseteq f^{-1}[G]$ iff $f[F] \subseteq G$. Therefore, $F \mathrel{\Vr(f)} G$ iff $G = f[F]$, and so $\Vr(f)$ is the function that maps $F \in \V(X_1)$ to $f[F] \in \V(X_2)$. Thus, $\Vr$ extends the Vietoris endofunctor on $\KHaus$. 

\eqref{t:Vr functor:item3}. It is well known (see \cite[Sec.~4]{Mic51}) that if $X$ is a Stone space, then $\V(X)$ is also a Stone space. Consequently, $\Vrbox, \Vrdia, \Vr$ restrict to $\StoneR$.
\end{proof}

\begin{remark}
The semi-functors $\Vrbox$ and $\Vrdia$ do not preserve identity. Indeed, if $R$ is the identity relation on $X$, then $F \Rboxrel G$ iff $G \subseteq F$ and $F \Rdiarel G$ iff $F \subseteq G$. Therefore, if $X \neq \varnothing$, then
neither $\Rbox$ nor $\Rdia$ is the identity on $\V(X)$ in $\KHausR$. 
\end{remark}

We recall (see, e.g., \cite[p.~74]{HV19}) that a \emph{dagger} on a category $\mathsf{C}$ is a contravariant functor $(-)^\dagger \colon \mathsf{C} \to \mathsf{C}$ that is the identity on objects and the composition $(-)^\dagger \circ (-)^\dagger$ is the identity functor on $\mathsf{C}$.
It is well known and easy to see that mapping a closed relation $R \colon X \to Y$ to its converse, which is clearly a closed relation, defines a dagger on $\KHausR$. 
We conclude the section by observing that $\Vrbox$ and $\Vrdia$ are definable from each other using $(-)^\dagger$, and that $\Vr$ commutes with $(-)^\dagger$. 

\begin{proposition}\label{l:Vr and dagger}
\hfill\begin{enumerate}[label=\normalfont(\arabic*), ref = \arabic*]
\item\label{l:Vr and dagger:item1} $\Vrbox \circ (-)^\dagger = (-)^\dagger \circ \Vrdia$ and $\Vrdia \circ (-)^\dagger = (-)^\dagger \circ \Vrbox$.
\item\label{l:Vr and dagger:item2} $\Vr \circ (-)^\dagger = (-)^\dagger \circ \Vr$.
\end{enumerate}
\end{proposition}

\begin{proof}
Since $\Vrbox, \Vrdia, \Vr$ coincide on objects and $(-)^\dagger$ fixes the objects, we only need to show that the compositions agree on the morphisms.

\eqref{l:Vr and dagger:item1}. We only prove the first equality since the second is proved similarly. Let $R \colon X \to Y$ be a morphism in $\KHausR$. For $F \in \V(X)$ and $G \in \V(Y)$ we have 
\begin{align*}
G \mathrel{\Vrbox(R^\dagger)} F &\iff F \subseteq R^\dagger[G] \iff F \subseteq R^{-1}[G] \iff F \mathrel{\Vrdia(R)} G \iff G \mathrel{\Vrdia(R)^\dagger} F.
\end{align*} 

\eqref{l:Vr and dagger:item2}.
Let $R \colon X \to Y$ be a morphism in $\KHausR$. If $F \in \V(X)$ and $G \in \V(Y)$, then \eqref{l:Vr and dagger:item1} implies 
\begin{align*}
G \mathrel{\Vr(R^\dagger)} F &\iff G \mathrel{\Vrbox(R^\dagger)} F \text{ and } G \mathrel{\Vrdia(R^\dagger)} F \iff G \mathrel{\Vrdia(R)^\dagger} F \text{ and } G \mathrel{\Vrbox(R)^\dagger} F\\
&\iff F \mathrel{\Vrdia(R)} G \text{ and } F \mathrel{\Vrbox(R)} G \iff G \mathrel{\Vr(R)^\dagger} F.\qedhere
\end{align*} 
\end{proof}

\begin{remark}\label{r:Vr and dagger functor}
We point out that \cref{l:Vr and dagger}\eqref{l:Vr and dagger:item1} simplifies the proof of \cref{lem:R R1 R2 closed}. Indeed, if $R \colon X \to Y$ is a closed relation, then $\Vrdia(R)=\Vrbox(R^\dagger)^\dagger$, and so is closed. Thus, it is enough to prove that $\Vrbox(R)$ is closed.
\end{remark}

\section{Closed relations and subordinations} \label{sec: subordinations}

In this section we recall the definition of subordination relations and the results connecting them to closed relations between compact Hausdorff spaces.
Subordinations on boolean algebras were introduced in \cite{BBSV17}. They are closely related to precontact relations \cite{DV06,DuentschVakarelov2007} and quasi-modal operators \cite{Celani2001}.

\begin{definition}\label{d:subordination} \cite[Def.~2.3]{BBSV17}
We call a relation $S \colon A \to B$ between boolean algebras a \emph{subordination relation} if it satisfies the following axioms, where $a,b\in A$ and $c,d\in B$:
\begin{enumerate}[label = (S\arabic*), ref = S\arabic*]
\item \label{S1} $0 \S 0$ and $1 \S 1$,
\item \label{S2} $a,b \S c$ implies $(a\vee b) \S c$, 
\item \label{S3} $a \S c,d$ implies $a \S (c\wedge d)$,  
\item \label{S4} $a\le b \S c\le d$ implies $a \S d$. 
\end{enumerate}
\end{definition}

Boolean algebras and subordination relations between them form a category, which we denote by $\BAS$. The identity on $B \in \BAS$ is the order $\le$ on $B$ and the composition is relation composition. We have the following generalization of Stone duality, which was obtained in \cite[Cor.~2.6]{ABC22a} by utilizing a result of Celani \cite[Thm.~4]{Cel18}. It is also a consequence of a more general result of Jung, Kurz, and Moshier \cite[Thm.~5.5.9]{KurzMoshierEtAl2023}.

\begin{theorem}\label{t:StoneR BAS equiv}
The categories $\StoneR$ and $\BAS$ are equivalent.
\end{theorem}

\begin{remark}\label{r:Clop and Uf}
The functors $\Clop \colon \StoneR \to \BAS$ and $\Uf \colon \BAS \to \StoneR$ yielding the above equivalence are defined as follows. On objects they act as the clopen and ultrafilter functors of Stone duality. On morphisms they act as follows. 
If $R \colon X \to Y$ is a closed relation between Stone spaces, then $\Clop(R)$ is the subordination relation $S_R \colon \Clop(X) \to \Clop(Y)$ given by 
\[
U \S_R V \iff R[U] \subseteq V.
\] 
If $S \colon A \to B$ is a subordination relation between boolean algebras, then $\Uf(S)$ is the closed relation $R_S \colon \Uf(A) \to \Uf(B)$ given by 
\[
x \mathrel{R_S} y \iff S[x] \subseteq y.
\]
\end{remark}

\begin{definition}\cite[p.~7]{ABC22a}
An \emph{$\mathsf{S5}$-subordination space} is a pair $(X, E)$ where $X$ is a Stone space and $E$ is a closed equivalence relation on $X$.  A closed relation $R \subseteq X_1 \times X_2$ between $\mathsf{S5}$-subordination spaces $(X_1,E_1)$ and $(X_2,E_2)$ is \emph{compatible} if $R \circ E_1 = R = E_2 \circ R$.
We let $\StoneER$ be the category of $\mathsf{S5}$-subordination spaces and compatible closed relations.
\end{definition}

\begin{theorem}{\cite[Cor.~3.12]{ABC22a}}\label{t:StoneER equiv KHausR}
The categories $\StoneER$ and $\KHausR$ are equivalent.
\end{theorem}

\begin{remark}\label{rem:G and Q}
The above equivalence is established by the functor $\Q \colon \StoneER \to \KHausR$ that maps $(X,E) \in \StoneER$ to the quotient space $X/E$ and a morphism $R \colon (X_1,E_1) \to (X_2,E_2)$ in $\StoneER$ to the closed relation $\Q(R)\coloneqq\pi_2 \circ R \circ \pi_1^\dagger$, where $\pi_1 \colon X_1 \to X_1/E_1$ and $\pi_2 \colon X_2 \to X_2/E_2$ are the projection maps.
A quasi-inverse of $\Q$ is defined by utilizing the well-known Gleason cover construction (see, e.g., \cite[Sec.~III.3]{Joh82}).
For a compact Hausdorff space $X$ let $\G(X)=(\widehat{X}, E)$ where $g_X \colon \widehat{X} \to X$ is the Gleason cover of $X$ and $x \E y$ iff $g_X(x)=g_X(y)$. For a closed relation $R \colon X \to Y$ let $\G(R) \colon \G(X) \to \G(Y)$ be given by $\G(R) \coloneqq g_Y^\dagger \circ R \circ g_X$.
This defines the Gleason cover functor $\G \colon \KHausR \to \StoneER$, which is a quasi-inverse of $\Q$ \cite[Thm.~4.6]{ABC22a}. 
\end{remark}

\begin{definition}
\hfill
\begin{enumerate}
\item We say that a subordination $S$ on a boolean algebra $B$ is an {\em $\mathsf{S5}$-subordination} if $S$ satisfies the following axioms, where $a,b \in B$:
\begin{enumerate}[label = (S\arabic*), ref = S\arabic*, start = 5]
\item \label{S5} $a \S b$ implies $a \le b$,
\item \label{S6} $a \S b$ implies $\neg b \S \neg a$,
\item \label{S7} $a \S b$ implies there is $c \in B$ such that $a \S c$ and $c \S b$.
\end{enumerate}
\item An {\em $\mathsf{S5}$-subordination algebra} is a pair $(B,S)$ where $B$ is a boolean algebra and $S$ is an $\mathsf{S5}$-subordination on $B$. 
\item  A subordination 
$T \subseteq B_1 \times B_2$  between $\mathsf{S5}$-subordination algebras
$(B_1,S_1)$ and $(B_2,S_2)$ is \emph{compatible} if $T \circ S_1 = T = S_2 \circ T$.
\item Let $\SubSfive$ be the category of $\mathsf{S5}$-subordination algebras and compatible subordinations.
\end{enumerate}
\end{definition}

De Vries algebras are special $\mathsf{S5}$-subordination algebras:

\begin{definition}
\hfill\begin{enumerate}
\item A \emph{de Vries algebra} is an $\mathsf{S5}$-subordination algebra $(B,S)$ such that $B$ is complete and $S$ satisfies
\begin{enumerate}[label = (S\arabic*), ref = S\arabic*, start = 8]
\item \label{S8} $b \neq 0$ implies that there is $a \neq 0$  such that $a \S b$.
\end{enumerate}
\item Let $\devS$ be the full subcategory of $\SubSfive$ consisting of de Vries algebras.
\end{enumerate}
\end{definition}

\begin{theorem}{\cite[Cor.~3.14, 4.7]{ABC22a}}\label{t:StoneER equiv SubSfive}
The categories $\KHausR$, $\StoneER$, $\SubSfive$, and $\devS$ are equivalent.
\end{theorem}

The equivalence of $\StoneER$ and $\SubSfive$ was obtained by utilizing the machinery of allegories and splitting equivalences. 
As we observed in the previous section, $\KHausR$ is a dagger category and $\Vr$ commutes with the dagger. In addition, ordering the hom-sets by inclusion turns $\KHausR$ into an order enriched category (meaning that each hom-set is partially ordered and composition preserves the order; see, e.g., \cite[p.~105]{ST14}).

An \emph{allegory} \cite{FS90,Joh02} is an order-enriched dagger category such that
\begin{enumerate}
\item each hom-set has binary meets,
\item $(-)^\dagger$ preserves the order on the hom-sets,
\item the modular law holds: $g f \wedge h \le (g \wedge h f^\dagger) f$  for all $f \colon C \to D$, $g \colon D \to E$, and $h \colon C \to E$.
\end{enumerate}

We already saw that $\Vr$ commutes with the dagger and it is straightforward to see that $\Vr$ preserves inclusions of relations. Therefore, $\Vr$ is a morphism of ordered categories with involution (see \cite{TsalenkoGisinEtAl1984,Lambek1999}).
However, as we will see in the next example, $\Vr$ does not preserve the meet, and hence is not a morphism of allegories (see \cite{Joh82,ABC22a}).

\begin{example}\label{ex:VR not morph allegories}
Since $\Vr$ preserves inclusions of relations, $\Vr(R_1 \cap R_2) \subseteq \Vr(R_1) \cap \Vr(R_2)$ for any pair of closed relations $R_1,R_2 \colon X \to Y$. We show that the other inclusion does not hold in general.
Let $X$ be a finite discrete space with at least two elements. We let $R_1 \colon X \to X$ be the identity and $R_2 \colon X \to X$ its complement.
Clearly $X$ is compact Hausdorff and $R_1,R_2$ are closed. Moreover, $R_1 \cap R_2 = \varnothing$. Thus, $X \mathrel{\Vr(R_1 \cap R_2)} X$ does not hold. However, $X=R_1^{-1}[X]=R_1[X]$ and $X=R_2^{-1}[X]=R_2[X]$ because $X$ has at least two elements. Therefore, $X \mathrel{(\Vr(R_1) \cap \Vr(R_2))} X$ holds. 
\end{example}

Our goal is to describe an endofunctor on $\SubSfive$ that corresponds to the Vietoris endofunctor $\V^{\sf R}$ on $\KHausR$, and then lift it to an endofunctor on $\devS$.
For this we first generalize the functor $\K \colon \ba \to \ba$ dual to $\V \colon \Stone \to \Stone$ (see the Introduction).

\section{The endofunctor \texorpdfstring{$\K$}{K} on \texorpdfstring{$\ba$}{BA}}

We recall the definition of the endofunctor $\K$ on $\ba$.
There are several equivalent approaches to defining $\K$. We will follow the one given in \cite{VV14}, which in turn relies on \cite{Joh85}.

\begin{definition}\label{def:KB}
	For $B\in\ba$ let $X=\{\Box_a\mid a\in B\}\cup\{\Diamond_a\mid a\in B\}$ be a set of symbols and define $\K(B)$ as the quotient of the free boolean algebra over $X$ by the relations
	\[
	\begin{array}{lll}
		\Box_1 = 1, & \quad & \Diamond_0 = 0, \\
		\Box_{a\wedge b} = \Box_a \wedge \Box_b, & \quad & \Diamond_{a\vee b} = \Diamond_a \vee \Diamond_b, \\
		\Box_{a\vee b} \le \Box_a \vee \Diamond_b, & \quad & \Box_a \wedge \Diamond_b \le \Diamond_{a\wedge b}.
	\end{array}
	\] 
\end{definition}

With a small abuse of notation, we write $\Box_a$ and $\Diamond_a$ also for the equivalence classes $[\Box_a]$ and $[\Diamond_a]$ in $\K(B)$.

\begin{remark}\label{rem:facts about KB}
	\hfill\begin{enumerate}
		\item \label{i:interdefinition}
			The last two inequalities in Definition~\ref{def:KB} can be replaced by the equation $\Box_a=\neg \Diamond_{\neg a}$, which is equivalent to $\Diamond_a=\neg \Box_{\neg a}$.
		Indeed, the two inequalities imply
		\begin{equation*}
			1=\Box_{1}=\Box_{a\vee \neg a} \le \Box_a \vee \Diamond_{\neg a} \quad \mbox{and} \quad \Box_a \wedge \Diamond_{\neg a} \le \Diamond_{a\wedge \neg a}=\Diamond_0=0.
		\end{equation*}
		Therefore, $\neg\Box_a=\Diamond_{\neg a}$, and hence $\Box_a=\neg \Diamond_{\neg a}$. Conversely, since 
		\begin{align*}
			\Box_{a\vee b} \wedge \Box_{\neg b}=\Box_{a\wedge \neg b} \le \Box_a \quad \mbox{and} \quad \Diamond_b \le \Diamond_{b \vee \neg a}= \Diamond_{a\wedge b} \vee \Diamond_{\neg a},
		\end{align*}
		it follows that
		\begin{align*}
			\Box_{a\vee b} \le \Box_a \vee \neg \Box_{\neg b} \quad \mbox{and} \quad \neg \Diamond_{\neg a} \wedge \Diamond_b \le \Diamond_{a\wedge b}.
		\end{align*}
		Thus, 
		\begin{align*}
			\Box_{a\vee b} \le \Box_a \vee \Diamond_b \quad \mbox{and} \quad \Box_a \wedge \Diamond_b \le \Diamond_{a\wedge b}.
		\end{align*}
		\item Equivalently, $\K(B)$ can be defined as the quotient of the free boolean algebra over the set $\{\Box_a\mid a\in B\}$ by the relations $\Box_1 = 1$ and $\Box_{a\wedge b} = \Box_a \wedge \Box_b$ (see \cite[Rem.~3.13]{KKV04} and \cite[Rem.~1]{VV14}). Alternatively, $\K(B)$ can be defined as the quotient of the free boolean algebra over the set $\{\Diamond_a\mid a\in B\}$ by the relations $\Diamond_0 = 0$ and $\Diamond_{a\vee b} = \Diamond_a \vee \Diamond_b$ (see \cite[Sec.~7]{Abr88} and \cite[Def.~2.4]{BK07}).
	\end{enumerate}
\end{remark}

\begin{definition}
	Let $A,B$ be boolean algebras and $\alpha\colon A \to B$ a boolean homomorphism. The map $\K(\alpha)\colon\K(A) \to \K(B)$ is the unique boolean homomorphism satisfying $\K(f)(\Box_a)=\Box_{\alpha(a)}$ and $\K(f)(\Diamond_a)=\Diamond_{\alpha(a)}$.
	
\end{definition}

It is straightforward to see that this defines an endofunctor on $\ba$ (see \cite[p.~122]{VV14}).

\begin{proposition}
	$\K$ is an endofunctor on $\ba$.
\end{proposition}

As suggested by the notation, there is a close connection between $\K(B)$ and modal operators on $B$.
We recall that a {\em modal operator} on $B$ is a function $\Box\colon B \to B$ satisfying $\Box 1 = 1$ and $\Box (a \wedge b)=\Box a \wedge \Box b$ (equivalently, $\Box$ preserves finite meets). A {\em modal algebra} is a pair $(B,\Box)$ where $\Box$ is a modal operator on $B$. Let $\ma$ be the category of modal algebras and modal homomorphisms (i.e.\ boolean homomorphisms preserving $\Box$).

\begin{remark}
	Modal algebras can alternatively be defined as pairs $(B,\Diamond)$ where $\Diamond 0=0$ and $\Diamond (a \vee b)=\Diamond a \vee \Diamond b$ (equivalently, $\Diamond$ preserves finite joins). The two operators $\Box$ and $\Diamond$ are interdefinable by the well-known identities $\Diamond a= \neg \Box \neg a$ and $\Box a=\neg \Diamond \neg a$.
\end{remark}

The next result is well known; see \cite[Fact~3]{VV14}. For the definition of algebras for an endofunctor see, e.g., \cite[Def.~5.37]{AHS06}.

\begin{theorem}\label{thm:VV}
	\begin{enumerate}[label=\normalfont(\arabic*), ref = \arabic*]
		\item[]
		\item \label{i:bij}
		For a boolean algebra $B$, there is a bijection between modal operators on $B$ and boolean homomorphisms $\K(B)\to B$.
		\item This bijection extends to an isomorphism between $\ma$ and the category $\Alg(\K)$ of algebras for the endofunctor $\K\colon \ba\to\ba$.
	\end{enumerate}
\end{theorem}

\begin{remark}\label{rem:modal operators}
	\hfill
	\begin{enumerate}
		\item The bijection of Theorem~\ref{thm:VV}\eqref{i:bij} is obtained by associating to each modal operator $\Box$ on $B$ the boolean homomorphism $\alpha\colon \K(B) \to B$ defined by $\alpha(\Box_a)=\Box a$; and to each boolean homomorphism $\alpha\colon \K(B) \to B$ the modal operator $\Box$ given by $\Box a=\alpha(\Box_a)$. 
		\item An analogous bijection holds if we replace $\Box$ by $\Diamond$.
	\end{enumerate}
\end{remark}

The next result is well known (see \cite{Abr88}, \cite{KKV04}, or \cite{VV14}). For the definition of coalgebras for an endofunctor see, e.g., \cite[Def.~140]{Ven07}.

\begin{theorem}
	The following diagram is commutative up to natural isomorphism. 
	\[
	\begin{tikzcd}[column sep=5pc, row sep=4pc]
		\Stone \arrow[r, shift left=1, "\Clop"] \arrow[d, "\V"'] & \ba \arrow[l, shift left=1, "\Uf"]  \arrow[d, "\K"] \\
		\Stone \arrow[r, shift left=1, "\Clop"]  & \ba \arrow[l, shift left=1, "\Uf"]
	\end{tikzcd}
	\]
	In other words, there are natural isomorphisms $\Clop\circ\V \simeq \K\circ\Clop$ and $\Uf\circ\K \simeq \V\circ\Uf$.
	Consequently, Stone duality extends to a dual equivalence between $\Alg(\K)$ and the category $\Coalg(\V)$ of coalgebras for the Vietoris endofunctor $\V\colon \Stone\to\Stone$. 
\end{theorem}

\section{The semi-functors \texorpdfstring{$\Ksbox$}{K{\textasciicircum}S\_Box} and \texorpdfstring{$\Ksdia$}{K{\textasciicircum}S\_Diamond}} \label{sec: K on BAS}

In this section we describe how to lift a subordination $S \colon A \to B$ to the subordinations $\Sbox, \Sdia \colon \K(A) \to \K(B)$ which are the algebraic counterparts of $\Rbox$ and $\Rdia$. 
In the next section we use $\Sbox$ and $\Sdia$ to define $\Sem$, which is the algebraic counterpart of $\Rem$.
To define $\Sbox$ and $\Sdia$, we first introduce the conjunctive and disjunctive normal forms of elements of $\K(B)$. 

\begin{definition}\label{d:disj conj normal form}
Let $x \in \K(B)$.
\begin{enumerate}[label=\normalfont(\arabic*), ref = \arabic*]
\item\label{d:disj conj normal form:item1} We say that $x$ is in \emph{disjunctive normal form} if it is written as a finite join
\[
x=\bigvee_{i=1}^n (\Box_{a_i} \wedge \Diamond_{b_{i1}} \wedge \dots \wedge \Diamond_{b_{i n_i}}),
\]
where $0 \neq  b_{ij} \le a_i$ for each $i,j$.
\item\label{d:disj conj normal form:item2} We say that $x$ is in \emph{conjunctive normal form} if it is written as a finite meet
\[
x=\bigwedge_{i=1}^n (\Diamond_{c_i} \vee \Box_{d_{i1}} \vee \dots \vee \Box_{d_{i n_i}}),
\]
where $c_i \le d_{ij} \neq 1$ for each $i,j$.
\end{enumerate}
Since we allow $n=0$, the empty join and meet yield disjunctive and conjunctive normal forms for $0$ and $1$, respectively.
\end{definition}

\begin{lemma}
Any element of $\K(B)$ can be written in disjunctive and conjunctive normal form. 
\end{lemma}

\begin{proof}
Let $x \in \K(B)$. We only prove that $x$ can be written in disjunctive normal form because the proof for the conjunctive normal form is similar.
Since $\K(B)$ is generated by the elements of the form $\Box_a$ and $\Diamond_b$, we may write $x \in \K(B)$ as a finite join of finite meets of $\Box_a$, $\Diamond_b$ or their negations (see, e.g., \cite[p.~14]{Sik69}). As $\neg \Box_a=\Diamond_{\neg a}$ and $\neg \Diamond_b=\Box_{\neg b}$, we have
\[
x=\bigvee_{i=1}^n (\Box_{a_{i 1}} \wedge \cdots \wedge \Box_{a_{i k_i}} \wedge \Diamond_{b_{i1}} \wedge \cdots \wedge \Diamond_{b_{i n_i}}).
\]
Let $a_i=a_{i 1} \wedge \cdots \wedge a_{i k_i}$. It follows from the definition of $\K(B)$ that
\[
\Box_{a_i}= \Box_{a_{i 1}} \wedge \cdots \wedge \Box_{a_{i k_i}}.
\]
Therefore,
\[
x=\bigvee_{i=1}^n (\Box_{a_i} \wedge \Diamond_{b_{i1}} \wedge \cdots \wedge \Diamond_{b_{i n_i}}).
\]
The definition of $\K(B)$ also yields that $\Box_a \wedge \Diamond_b=\Box_a \wedge \Diamond_{a \wedge b}$ for each $a,b \in B$. Thus, by replacing each $b_{ij}$ with $a_i \wedge b_{ij}$, we may assume that $b_{ij} \le a_i$ for each $i,j$.
Since $\Diamond 0 = 0$, we can suppress every disjunct $\Box_{a_i} \wedge \Diamond_{b_{i1}} \wedge \cdots \wedge \Diamond_{b_{i n_i}}$ with $b_{ij}=0$ for some $j$. 
\end{proof}

The next definition is motivated by \cref{prop:description-on-morphisms}.

\begin{definition}\label{def:proximity on KB-12}
Let $S \colon A \to B$ be a subordination, $x\in\K(A)$, and $y \in \K(B)$.
\begin{enumerate}
\item We set $x \Sboxrel y$ if it is possible to write $x$ in disjunctive normal form and $y$ in conjunctive normal form
\begin{equation*}
x=\bigvee_{i=1}^n (\Box_{a_i} \wedge \Diamond_{b_{i1}} \wedge \cdots \wedge \Diamond_{b_{i n_i}}), \quad y=\bigwedge_{j=1}^m (\Diamond_{c_j} \vee \Box_{d_{j1}} \vee \cdots \vee \Box_{d_{j m_j}}),
\end{equation*}
so that for each $i \le n$ and $j \le m$ there exists $k \le m_j$ with $a_i \S d_{jk}$.
\item We set $x \Sdiarel y$ if it is possible to write $x$ in disjunctive normal form and $y$ in conjunctive normal form
\begin{equation*}
x=\bigvee_{i=1}^n (\Box_{a_i} \wedge \Diamond_{b_{i1}} \wedge \cdots \wedge \Diamond_{b_{i n_i}}), \quad y=\bigwedge_{j=1}^m (\Diamond_{c_j} \vee \Box_{d_{j1}} \vee \cdots \vee \Box_{d_{j m_j}}),
\end{equation*}
so that for each $i \le n$ and $j \le m$ there exists $l \le n_i$ such that $b_{il} \mathrel{S} c_j$.
\end{enumerate}
\end{definition}

The following technical lemma characterizes the order on $\K(B)$ in terms of normal forms.

\begin{lemma}\label{lem:le in KB}
Let $x,y \in \K(B)$ be written in disjunctive and conjunctive normal form
\begin{equation*}
x=\bigvee_{i=1}^n (\Box_{a_i} \wedge \Diamond_{b_{i1}} \wedge \cdots \wedge \Diamond_{b_{i n_i}}), \quad y=\bigwedge_{j=1}^m (\Diamond_{c_j} \vee \Box_{d_{j1}} \vee \cdots \vee \Box_{d_{j m_j}}).
\end{equation*}
Then $x \le y$ iff for each $i\le n$ and $j\le m$ there exists $k \le m_j$ such that $a_i \le d_{jk}$ or there exists $l \le n_j$ such that $b_{il} \le c_j$.
\end{lemma}

\begin{proof}
	The inequality
	\[
	\bigvee_{i=1}^n (\Box_{a_i} \wedge \Diamond_{b_{i1}} \wedge \cdots \wedge \Diamond_{b_{i n_i}}) \le \bigwedge_{j=1}^m (\Diamond_{c_j} \vee \Box_{d_{j1}} \vee \cdots \vee \Box_{d_{j m_j}})
	\]
	holds iff for each $i, j$
	\[
	\Box_{a_i} \wedge \Diamond_{b_{i1}} \wedge \cdots \wedge \Diamond_{b_{i n_i}} \le \Diamond_{c_j} \vee \Box_{d_{j1}} \vee \cdots \vee \Box_{d_{j m_j}}.
	\]
	Thus, it is sufficient to show that
	for $a$, $b_1, \ldots, b_n$, $c$, $d_1, \ldots, d_m \in B$ such that $b_i \le a$ and $c \le d_j$ for every $i,j$, the following two conditions are equivalent:
	\begin{enumerate}
		\item \label{i:inequality}
		$\Box_a \wedge \Diamond_{b_1} \wedge \dots \wedge \Diamond_{b_n} \le \Diamond_c \vee \Box_{d_1} \vee \dots \vee \Box_{d_m}$,
		\item \label{i:existence}
		there exists $j$ such that $a \le d_j$ or there exists $i$ such that $b_i \le c$.
	\end{enumerate}
	
	\eqref{i:existence} $\Rightarrow$ \eqref{i:inequality}. It follows from Definition~\ref{def:KB} that $a \le d_j$ implies $\Box_a \le \Box_{d_j}$ and $b_i \le c$ implies $\Diamond_{b_i} \le \Diamond_{c}$. Thus, \eqref{i:inequality} holds.
	
	\eqref{i:inequality} $\Rightarrow$ \eqref{i:existence}.
	Suppose \eqref{i:inequality} holds, so $\Box_a \wedge \Diamond_{b_1} \wedge \dots \wedge \Diamond_{b_n} \leq \Diamond_c \vee \Box_{d_1} \vee \dots \vee \Box_{d_m}$.
	Let $A = \{0, 1\}$ be the two-element Boolean algebra.
	Define two functions $\Diamond, \Box \colon B \to A$ by	
	\[
	\Diamond e =
	\begin{cases}
		0 & \text{if } e \wedge (a \wedge \neg c)=0, \\
		1 & \text{if } e \wedge (a \wedge \neg c) \neq 0
	\end{cases}
	\]
	and
	\[
	\Box e =
	\begin{cases}
		1 & \text{if } \lnot e \wedge (a \wedge \neg c)=0, \\
		0 & \text{if } \lnot e \wedge (a \wedge \neg c) \neq 0
	\end{cases}
	\]
	for each $e \in B$.
	It is straightforward to see that $\Diamond$ preserves finite joins and that $\Box e = \lnot \Diamond \lnot e$ for each $e \in B$.
	Therefore, there is a unique boolean homomorphism $\alpha \colon \K(B) \to A$ such that $\alpha(\Box_e)=\Box e$ and $\alpha(\Diamond_e)=\Diamond e$ for each $e \in B$. Therefore,  
	\begin{align*}
		\alpha(\Box_a \wedge \Diamond_{b_1} \wedge \dots \wedge \Diamond_{b_n}) &=\Box a \wedge \Diamond b_1 \wedge \dots \wedge \Diamond b_n,\\
		\alpha(\Diamond_c \vee \Box_{d_1} \vee \dots \vee \Box_{d_m}) &=\Diamond c \vee \Box d_1 \vee \dots \vee \Box d_m.
	\end{align*}
	Since $\Box_a \wedge \Diamond_{b_1} \wedge \dots \wedge \Diamond_{b_n} \leq \Diamond_c \vee \Box_{d_1} \vee \dots \vee \Box_{d_m}$ and $\alpha$ is order preserving,
	\[
	\Box a \wedge \Diamond {b_1} \wedge \dots \wedge \Diamond {b_n} \leq \Diamond c \vee \Box {d_1} \vee \dots \vee \Box {d_m}.
	\]
	We have $\Box a = 1$ because $\lnot a \land (a \land \lnot c) = 0$, and $\Diamond c = 0$ because $c \land (a \land \lnot c) = 0$.
	Thus, 
	\[
	\Diamond {b_1} \wedge \dots \wedge \Diamond {b_n} \leq \Box {d_1} \vee \dots \vee \Box {d_m}.
	\]
	Consequently, $\Diamond {b_1} \wedge \dots \wedge \Diamond {b_n} = 0$ or $\Box {d_1} \vee \dots \vee \Box {d_m} = 1$.
	In the former case, there is $i \leq n$ such that $\Diamond b_i = 0$, so $b_i \land (a \land \lnot c) = 0$, and hence $b_i \land \lnot c = 0$ (since	$b_i \leq a$). Therefore,
	$b_i \leq c$.
	In the latter case, there is $j \leq m$ such that $\Box d_j = 1$, so $\lnot d_j \land (a \land \lnot c) = 0$, and hence $\lnot d_j \land a = 0$ (since $c \leq d_j$). Thus, $a \leq d_j$.
\end{proof}

The next proposition says that the definitions of $\Sbox$ and $\Sdia$ are independent of the normal forms used to write $x$ and $y$. This fact will be useful in proving that $\Sbox$ and $\Sdia$ are subordinations in \cref{th:Sbox Sdia subordinations}.

\begin{proposition}\label{p:Sbox Sdia independent representatives}
Let $S \colon A \to B$ be a subordination. If $x \in \K(A)$ and $y \in \K(B)$ are written in disjunctive and conjunctive normal form
\begin{equation*}
x=\bigvee_{i=1}^n (\Box_{a_i} \wedge \Diamond_{b_{i1}} \wedge \cdots \wedge \Diamond_{b_{i n_i}}), \quad y=\bigwedge_{j=1}^m (\Diamond_{c_j} \vee \Box_{d_{j1}} \vee \cdots \vee \Box_{d_{j m_j}}),
\end{equation*}
then
\begin{enumerate}[label=\normalfont(\arabic*), ref = \arabic*]
\item\label{p:Sbox Sdia independent representatives:item1} $x \Sboxrel y$ iff for all $i,j$ there exists $k \le m_j$ such that $a_i \S d_{jk}$.
\item\label{p:Sbox Sdia independent representatives:item2} $x \Sdiarel y$ iff for all $i,j$ there exists $l \le n_i$ such that $b_{il} \S c_j$.
\end{enumerate}
\end{proposition}

\begin{proof}
We only prove \eqref{p:Sbox Sdia independent representatives:item1} since the proof of \eqref{p:Sbox Sdia independent representatives:item2} is similar.
The right-to-left implication follows immediately from the definition of $\Sbox$. To prove the left-to-right implication, assume that $x \Sboxrel y$. The definition of $\Sbox$ implies that it is possible to write $x$ and $y$ in conjunctive and disjunctive normal form 
\begin{equation*}
x=\bigvee_{r=1}^{t} (\Box_{e_r} \wedge \Diamond_{f_{r1}} \wedge \cdots \wedge \Diamond_{f_{r t_r}}), \quad y=\bigwedge_{s=1}^{u} (\Diamond_{g_s} \vee \Box_{h_{s1}} \vee \cdots \vee \Box_{h_{s u_s}}),
\end{equation*}
so that for any $r,s$ there exists $q \le u_s$ with $e_r \S h_{s q}$.
Fix $i \le n$ and $j \le m$. We show that there is $k \le m_j$ such that $a_i \S d_{jk}$.
Clearly
\begin{align}\label{eq:le two normal forms}
\begin{split}
\Box_{a_i} \wedge \Diamond_{b_{i1}} \wedge \cdots \wedge \Diamond_{b_{i n_i}} \le \bigvee_{r=1}^{t} (\Box_{e_r} \wedge \Diamond_{f_{r1}} \wedge \cdots \wedge \Diamond_{f_{r t_r}}),\\
\bigwedge_{s=1}^{u} (\Diamond_{g_s} \vee \Box_{h_{s1}} \vee \cdots \vee \Box_{h_{s u_s}}) \le \Diamond_{c_j} \vee \Box_{d_{j1}} \vee \cdots \vee \Box_{d_{j m_j}}.
\end{split}
\end{align}
Thus,
\begin{equation*}
\Box_{a_i} \wedge \Diamond_{b_{i1}} \wedge \cdots \wedge \Diamond_{b_{i n_i}} \le \bigvee_{r=1}^{t} (\Box_{e_r} \wedge \Diamond_{f_{r1}} \wedge \cdots \wedge \Diamond_{f_{r t_r}}) \le \Diamond_0 \vee \Box_{e_1} \vee \dots \vee \Box_{e_{t}}.
\end{equation*}
We claim that there is $r'\le t$ such that $a_i \le e_{r'}$.
If $e_r =1$ for some $r$, then we can take $r' = r$.
Otherwise, $e_r \neq 1$ for every $r$, and hence the expression $\Diamond_0 \vee \Box_{e_1} \vee \dots \vee \Box_{e_{t}}$ is in conjunctive normal form (with a unique conjunct). Therefore, since $a_i\ne 0$, \Cref{lem:le in KB} implies that there is $r'\le t$ such that $a_i \le e_{r'}$, proving our claim.

By our assumption, for each $s \le u$ there exists $q_s$ such that $e_{r'} \S h_{s q_s}$. Consequently, $e_{r'} \S \bigwedge_{s=1}^{u} h_{s q_s}$ and \Cref{eq:le two normal forms} yields
\begin{equation*}
\Box_{\bigwedge_{s=1}^{u} h_{s q_s}} = \Box_{h_{1 q_1}} \wedge \dots \wedge \Box_{h_{u q_{u}}} \le \bigwedge_{s=1}^{u} (\Diamond_{g_s} \vee \Box_{h_{s1}} \vee \cdots \vee \Box_{h_{s u_s}}) \le \Diamond_{c_j} \vee \Box_{d_{j1}} \vee \cdots \vee \Box_{d_{j m_j}}.
\end{equation*}
By \Cref{lem:le in KB}, there exists $k \le m_j$ such that $\bigwedge_{s=1}^{u} h_{s q_s} \le d_{jk}$. Thus, we have found $k \le m_j$ such that
\begin{equation*}
a_i \le e_{r'} \S \bigwedge_{s=1}^{u} h_{s q_s} \le d_{jk},
\end{equation*} 
and so $a_i \S d_{jk}$.
\end{proof}

\begin{theorem}\label{th:Sbox Sdia subordinations}
If $S \colon A \to B$ is a subordination, then $\Sbox$ and $\Sdia$ are subordinations.
\end{theorem}

\begin{proof}
We only prove that $\Sbox$ is a subordination because the proof for $\Sdia$ is similar.
For this we must show that $\Sbox$ satisfies the axioms \eqref{S1}--\eqref{S4} of \Cref{d:subordination}. 

\eqref{S1}. Write $0$ in disjunctive normal form as the empty join and in conjunctive normal form as $0=\Diamond_0$. Then $0 \Sboxrel 0$ holds trivially. That $1 \Sboxrel 1$ is proved similarly.

\eqref{S2}. Suppose $x \Sboxrel z$ and $y \Sboxrel z$. Write $x,y$ in disjunctive normal form and $z$ in conjunctive normal form 
\begin{align*}
x &=\bigvee_{i=1}^n (\Box_{a_i} \wedge \Diamond_{b_{i 1}} \wedge \cdots \wedge \Diamond_{b_{i n_i}}), \quad y =\bigvee_{j=1}^m (\Box_{c_j} \wedge \Diamond_{d_{j 1}} \wedge \cdots \wedge \Diamond_{d_{j m_j}}),\\
z &=\bigwedge_{r=1}^t (\Diamond_{e_r} \vee \Box_{f_{r 1}} \vee \cdots \vee \Box_{f_{r t_r}}).
\end{align*}
By \Cref{p:Sbox Sdia independent representatives}\eqref{p:Sbox Sdia independent representatives:item1}, for all $i,r$ there exists $k \le t_r$ such that $a_i \S f_{r k}$ and for all $j,r$ there exists $l \le t_r$ such that $c_j \S f_{r l}$. Since $x \vee y$ can be written in disjunctive normal form as 
\begin{equation*}
x \vee y =\bigvee_{i=1}^n (\Box_{a_i} \wedge \Diamond_{b_{i 1}} \wedge \cdots \wedge \Diamond_{b_{i n_i}}) \vee \bigvee_{j=1}^m (\Box_{c_j} \wedge \Diamond_{d_{j 1}} \wedge \cdots \wedge \Diamond_{d_{j m_j}}),
\end{equation*}
it follows from the definition of $\Sbox$ that $(x \vee y) \Sboxrel z$.

The proof of \eqref{S3} is similar to that of \eqref{S2}.

\eqref{S4}. Let $x \le y \Sboxrel z \le w$. By the definition of $\Sbox$, we can write $x,y$ in disjunctive normal form and $z,w$ in conjunctive normal form
\begin{align*}
& x=\bigvee_{i=1}^n (\Box_{a_i} \wedge \Diamond_{b_{i1}} \wedge \cdots \wedge \Diamond_{b_{i n_i}}), \quad y=\bigvee_{r=1}^{t} (\Box_{e_r} \wedge \Diamond_{f_{r1}} \wedge \cdots \wedge \Diamond_{f_{r t_r}}),\\
& z=\bigwedge_{s=1}^{u} (\Diamond_{g_s} \vee \Box_{h_{s1}} \vee \cdots \vee \Box_{h_{s u_s}}), \quad w=\bigwedge_{j=1}^m (\Diamond_{c_j} \vee \Box_{d_{j1}} \vee \cdots \vee \Box_{d_{j m_j}}),
\end{align*}
so that for all $r,s$ there exists $q \le u_s$ with $e_r \S h_{s q}$. By arguing as in the proof of \Cref{p:Sbox Sdia independent representatives} we obtain that for any $i,j$ there exists $k \le m_j$ such that $a_i \S d_{jk}$. Therefore, $x \Sboxrel w$ by \Cref{p:Sbox Sdia independent representatives}\eqref{p:Sbox Sdia independent representatives:item1}. 
\end{proof}

We are ready to define the semi-functors $\Ksbox,\Ksdia \colon \BAS \to \BAS$.

\begin{definition}\label{d:Ksbox Ksdia}
For a boolean algebra $B$, let $\Ksbox(B) = \Ksdia(B) = \K(B)$. For a morphism $S$ in $\BAS$, let
$\Ksbox(S)=\Sbox$ and $\Ksdia(S)=\Sdia$.
\end{definition}

\begin{theorem}\label{t:Ksbox Ksdia Ks (semi-)functors}
$\Ksbox,\Ksdia \colon \BAS \to \BAS$ are semi-functors.
\end{theorem}

\begin{proof}
We show that $\Ksbox$ is a semi-functor. That $\Ksdia$ is a semi-functor is proved similarly.
Let $S \colon B_1 \to B_2$ and $T \colon B_2 \to B_3$ be subordinations, $x \in \Ks(B_1)$, and $y \in \Ks(B_3)$. Suppose that $x$ and $y$ are written in disjunctive and conjunctive normal form
\begin{align*}
x=\bigvee_{i=1}^n (\Box_{a_i} \wedge \Diamond_{b_{i1}} \wedge \cdots \wedge \Diamond_{b_{i n_i}}), \quad y=\bigwedge_{j=1}^m (\Diamond_{c_j} \vee \Box_{d_{j1}} \vee \cdots \vee \Box_{d_{j m_j}}).
\end{align*}
We show that $x \mathrel{\Ksbox(T \circ S)} y$ iff $x \left(\mathrel{\Ksbox(T)} \circ \mathrel{\Ksbox(S)}\right) y$. To prove the left-to-right implication, suppose that $x \mathrel{\Ksbox(T \circ S)} y$, which means that for every $i \le n$ and $j \le m$ there exists $k_j \le m_j$ such that $a_i \mathrel{(T \circ S)} d_{jk_j}$. Thus, there is $f_{ij} \in B_2$ such that $a_i \S f_{ij} \T d_{jk_j}$, and so $\Box_{a_i} \mathrel{\Ksbox(S)} \Box_{f_{ij}} \mathrel{\Ksbox(T)} \Box_{d_{j k_j}}$. It follows from the properties of subordinations that for each $i$ we have
\[
\Box_{a_i} \mathrel{\Ksbox(S)} \left(\bigwedge_{j=1}^m \Box_{f_{ij}}\right) \mathrel{\Ksbox(T)} \left(\bigwedge_{j=1}^m \Box_{d_{jk_j}}\right),
\]
and hence
\[
x \le \left(\bigvee_{i=1}^n \Box_{a_i}\right) \mathrel{\Ksbox(S)} \left(\bigvee_{i=1}^n \bigwedge_{j=1}^m \Box_{f_{ij}}\right) \mathrel{\Ksbox(T)} \left(\bigwedge_{j=1}^m \Box_{d_{jk_j}}\right) \le y.
\]
Let $z = \bigvee_{i=1}^n \bigwedge_{j=1}^m \Box_{f_{ij}} \in \Ks(B_2)$. Then $x \mathrel{\Ksbox(S)} z \mathrel{\Ksbox(T)} y$. This shows that $\Ksbox(T \circ S) \subseteq \Ksbox(T) \circ \Ksbox(S)$. To prove the other inclusion, suppose that $x \mathrel{\left(\Ksbox(T) \circ \Ksbox(S)\right)} y$, and that $x,y$ are written in disjunctive and conjunctive normal form as above. Then there is $z \in \Ks(B_2)$ such that $x \mathrel{\Ksbox(S)} z \mathrel{\Ksbox(T)} y$. Write $z$ in conjunctive normal form 
\[
z = \bigwedge_{r=1}^t (\Diamond_{e_r} \vee \Box_{f_{r1}} \vee \cdots \vee \Box_{f_{r t_r}}).
\]
Fix $i \le n$ and $j \le m$. Since $x \mathrel{\Ksbox(S)} z$, for each $r \le t$ there is $l_r \le t_r$ such that $a_i \S f_{r l_r}$. Because
\[
\Box_{\bigwedge_{r=1}^t f_{r l_r}} \le z \mathrel{\Ksbox(T)} y, 
\]
there exists $s \le m_j$ such that $\left(\bigwedge_{r=1}^t f_{r l_r}\right) \T d_{js}$. Then $a_i \S \left(\bigwedge_{k=1}^t f_{k l_k}\right) \T d_{js}$, and so $a_i \mathrel{(T \circ S)} d_{js}$. Therefore, $x \mathrel{\Ksbox(T \circ S)} y$. This shows that $\Ksbox$ is a semi-functor.
\end{proof}

\section{The endofunctor \texorpdfstring{$\Ks$}{K{\textasciicircum}S} on \texorpdfstring{$\BAS$}{BA{\textasciicircum}S}}

In this section we utilize $\Sbox$ and $\Sdia$ defined in the previous section to lift a subordination $S \colon A \to B$ to the subordination $\Sem \colon \K(A) \to \K(B)$, which is the algebraic counterpart of $\Rem$.
This allows us to extend the endofunctor $\K \colon \ba \to \ba$ to the endofunctor $\Ks \colon \BAS \to \BAS$.

For boolean algebras $A$ and $B$, we denote by $\BAS(A,B)$ the poset of subordinations $S \colon A \to B$ ordered by inclusion.\footnote{This is in contrast to \cite{ABC22a}, where $\BAS(A,B)$ was ordered by reverse inclusion to guarantee that $\BAS$ was an allegory.}
Our goal is to define $\Sem$ as the join of $\Sbox$ and $\Sdia$ in $\BAS(\K(A),\K(B))$.
For this we need to show that joins exist in $\BAS(\K(A),\K(B))$. In fact, we will show that $\BAS(A,B)$ is always a frame, where we recall (see, e.g., \cite[p.~10]{PP12}) that a complete lattice $L$ is a \emph{frame} if it satisfies $a \wedge \bigvee E = \bigvee \{ a \wedge b \mid b \in E \}$ for each $a \in L$ and $E \subseteq L$.

We also recall (see, e.g., \cite[p.~10]{PP12}) that a complete lattice is a \emph{coframe} if its order-dual is a frame.
Let $X$ be the Stone dual of $A$ and $Y$ the Stone dual of $B$. It is clear that $(\StoneR(X,Y), \subseteq)$ is the poset of closed subsets of $X \times Y$, and so is a coframe.
By \cite[Thm.~2.14]{ABC22a}, $(\BAS(A,B), \subseteq)$ is dually isomorphic to $(\StoneR(\Uf(A),\Uf(B)), \subseteq)$. Therefore, $\BAS(A,B)$ is a frame.
However, this proof uses Stone duality. As promised in \cite[Rem.~2.15]{ABC22a}, we give a choice free proof of this result. 

\begin{theorem}\label{lem:join subordinations}
$\BAS(A,B)$ is a frame, where meets are given by intersections and the join of $\{ S_\alpha \} \subseteq \BAS(A,B)$ is given by $x \left(\mathrel{\bigvee S_\alpha}\right) y$ iff there exist finite subsets $F \subseteq A$ and $G \subseteq B$ such that $x=\bigvee F$, $y= \bigwedge G$ and for all $a \in F$ and $b \in G$ there is $\alpha$ with $a \mathrel{S_\alpha} b$.
\end{theorem}

\begin{proof}
It is straightforward to see that the intersection of a family of subordinations is a subordination. Therefore, $\BAS(A,B)$ is a complete lattice, where meets are given by intersections.
Let $\{ S_\alpha \} \subseteq \BAS(A,B)$. Define $S \colon A \to B$ by $x \S y$ iff there exist finite subsets $F \subseteq A$ and $G \subseteq B$ such that $x=\bigvee F$, $y= \bigwedge G$ and for all $a \in F$ and $b \in G$ there is $\alpha$ with $a \mathrel{S_\alpha} b$.
It is clear that $S$ contains $S_\alpha$ for each $\alpha$. We show that $S$ is a subordination.

\eqref{S1}. Since $0=\bigvee \varnothing$ and $1=\bigwedge \varnothing$, it trivially holds that $0 \S 0$ and $1 \S 1$.

\eqref{S2}. Suppose that $x \S y$ and $x' \S y$. Then there exist $F,F' \subseteq A$ and $G,G' \subseteq B$ such that $x=\bigvee F$, $x'=\bigvee F'$, $y=\bigwedge G = \bigwedge G'$ and for all $a \in F$ and $b \in G$ there exists $\alpha$ with $a \mathrel{S_\alpha} b$ and for all $a' \in F'$ and $b' \in G'$ there exists $\beta$ with $a' \mathrel{S_\beta} b'$. Since $B$ is a boolean algebra, by distributivity we obtain that $y= \bigwedge G''$, where $G''= \{ b \vee b' \mid b \in G \text{ and } b' \in G' \}$. Let $F''=F \cup F'$. Then $x \vee x' =\bigvee F''$. Let $a'' \in F''$ and $b'' \in G''$. Then $b''=b \vee b'$ with $b \in G$ and $b' \in G'$. Therefore, if $a'' \in F$, there is $\alpha$ such that $a'' \S_\alpha b \le b \vee b'$. If $a'' \in F'$, there is $\beta$ such that $a'' \S_\beta b' \le b \vee b'$. In either case, $a'' \S b''$. Thus, $x \S y$.

\eqref{S3} is proved similarly to \eqref{S2}.

\eqref{S4}. Suppose $x' \le x \mathrel{S} y \le y'$. Then $x=\bigvee F$, $y= \bigwedge G$, and for all $a \in F$ and $b \in G$ there is $\alpha$ with $a \mathrel{S_\alpha} b$. Let $F' = \{ a \wedge x' \mid a \in F\}$ and $G' = \{ b \vee y' \mid b \in G\}$. Then $x'=\bigvee F'$ and $y' = \bigwedge G'$. Moreover, if $a' \in F'$ and $b' \in G'$, there exist $a \in F$ and $b \in G$ such that $a'=a \wedge x'$ and $b'= b \vee y'$. Therefore, $a'=a \wedge x \le a \S_\alpha b \le b \vee y'=b'$. Thus, $x' \S y'$.

It remains to prove that $S$ is the least upper bound of $\{ S_\alpha \}$. Let $T \colon A \to B$ be such that $S_\alpha \subseteq T$ for each $\alpha$ and let $x \S y$. Then $x=\bigvee F$, $y=\bigwedge G$ for some finite $F$ and $G$ such that for all $a \in F$ and $b \in G$ there is $\alpha$ with $a \S_\alpha b$. Therefore, $a \T b$. Since $T$ satisfies \eqref{S2} and \eqref{S3}, it follows that $x = \left(\bigvee F \right) \T \left( \bigwedge G \right) = y$. Thus, $S \subseteq T$.

It is left to show that $T \cap \bigvee S_\alpha = \bigvee (T \cap S_\alpha)$ for each $T \in \BAS(A,B)$ and $\{ S_\alpha \} \subseteq \BAS(A,B)$. The right-to-left inclusion is clear since $\BAS$ is a complete lattice. To show the left-to-right inclusion, let $x \in A$ and $y \in B$ such that $x \mathrel{\left(T \cap \bigvee S_\alpha \right)} y$. Then $x \T y$ and $x \left(\mathrel{\bigvee S_\alpha}\right) y$. So there exist finite subsets $F \subseteq A$ and $G \subseteq B$ such that $x=\bigvee F$, $y= \bigwedge G$ and for all $a \in F$ and $b \in G$ there is $\alpha$ with $a \mathrel{S_\alpha} b$. For all $a \in F$ and $b \in G$ we have $a \le x \T y \le b$, which implies that $a \T b$. Therefore, for all $a \in F$ and $b \in G$ there is $\alpha$ with $a \mathrel{(T \cap S_\alpha)} b$, and hence $x \mathrel{\bigvee (T \cap S_\alpha)} y$. Thus, $\BAS(A,B)$ is a frame.
\end{proof}

\begin{definition}\label{def:Sem}
Let $S \colon A \to B$ be a subordination between boolean algebras. We define $\Sem$ to be the join $\Sbox \vee \Sdia$ in $\BAS(\K(A),\K(B))$. 
\end{definition}

We thus have the following description of $\Sem$, which is an immediate consequence of \Cref{def:proximity on KB-12,th:Sbox Sdia subordinations,lem:join subordinations}.

\begin{corollary}\label{c:Sem description}
$\Sem$ is a subordination and $x \Semrel y$ iff it is possible to write $x$ and $y$ in disjunctive and conjunctive normal form
\begin{equation*}
x=\bigvee_{i=1}^n (\Box_{a_i} \wedge \Diamond_{b_{i1}} \wedge \cdots \wedge \Diamond_{b_{i n_i}}), \quad y=\bigwedge_{j=1}^m (\Diamond_{c_j} \vee \Box_{d_{j1}} \vee \cdots \vee \Box_{d_{j m_j}})
\end{equation*}
so that for all $i,j$ there exists $k \le m_j$ with $a_i \S d_{jk}$ or there exists $l \le n_i$ with $b_{i l} \S c_j$. 
\end{corollary}

The description of $\Sem$ given in \Cref{c:Sem description} is similar
to \Cref{def:proximity on KB-12}. In \Cref{p:Sbox Sdia independent representatives} we showed that the definitions of $\Sbox$ and $\Sdia$ are independent of the normal forms used to represent elements. The same is true for $\Sem$ and the proof is similar, so we skip it.

\begin{proposition}\label{p:Sem independent representatives}
Let $S \colon A \to B$ be a subordination. If $x \in \K(A)$ and $y \in \K(B)$ are written in disjunctive and conjunctive normal form
\begin{equation*}
x=\bigvee_{i=1}^n (\Box_{a_i} \wedge \Diamond_{b_{i1}} \wedge \cdots \wedge \Diamond_{b_{i n_i}}), \quad y=\bigwedge_{j=1}^m (\Diamond_{c_j} \vee \Box_{d_{j1}} \vee \cdots \vee \Box_{d_{j m_j}}),
\end{equation*}
then $x \Semrel y$ iff for all $i,j$ there exists $k \le m_j$ with $a_i \S d_{jk}$ or there exists $l \le n_i$ with $b_{i l} \S c_j$. 
\end{proposition}

\begin{definition}\label{d:Ks}
For a boolean algebra $B$, let $\Ks(B)=\K(B)$; and for a morphism $S$ in $\BAS$, let $\Ks(S)=\Sem$.
\end{definition}

\begin{theorem}\label{t:Ks functor}
$\Ks$ is an endofunctor on $\BAS$.
\end{theorem}

\begin{proof}
We show that $\Ks$ preserves composition; that is, $\Ks(T \circ S) = \Ks(T) \circ \Ks(S)$. We first show the left-to-right inclusion.
By \Cref{def:Sem}, $\Ksbox(S), \Ksdia(S) \subseteq \Ks(S)$ and $\Ksbox(T), \Ksdia(T) \subseteq \Ks(T)$. It then follows that 
\[
(\Ksbox(T) \circ \Ksbox(S)) \subseteq (\Ks(T) \circ \Ks(S)) \quad \text{and} \quad (\Ksdia(T) \circ \Ksdia(S)) \subseteq (\Ks(T) \circ \Ks(S)).
\]
Since $\Ksbox$ and $\Ksdia$ are semi-functors, we have
\begin{align*}
\Ks(T \circ S) & = \Ksbox(T \circ S) \vee \Ksdia(T \circ S) = (\Ksbox(T) \circ \Ksbox(S)) \vee (\Ksdia(T) \circ \Ksdia(S)).
\end{align*}
Thus, $\Ks(T \circ S) \subseteq \Ks(T) \circ \Ks(S)$. To prove the other inclusion, let $x \in \Ks(B_1)$ and $y \in \Ks(B_3)$ be written in disjunctive and conjunctive normal form as above, and assume that $x \left(\Ks(T) \circ \Ks(S)\right) y$. Then there exists $z \in \Ks(B_2)$ such that $x \mathrel{\Ks(S)} z \mathrel{\Ks(T)} y$. Write $z$ in conjunctive normal form
\[
z = \bigwedge_{r=1}^t (\Diamond_{e_r} \vee \Box_{f_{r1}} \vee \cdots \vee \Box_{f_{r t_r}}).
\]
Fix $i \le n$ and $j \le m$. We can reorder the conjuncts in the conjunctive normal form of $z$ in such a way that there is $t' \le t$ such that $r \le t'$ implies that there is $l_r \le t_r$ with $a_i \S f_{r l_r}$ and $r > t'$ implies that there is $s \le n_i$ such that $b_{i s} \S e_r$. We also allow $t'=0$ for the case in which the first condition never holds. Let $f = \bigwedge_{r=1}^{t'} f_{r l_r}$. Then
\begin{align*}
\Box_f \wedge \Diamond_{(e_{t'+1} \wedge f)} \wedge \dots \wedge \Diamond_{(e_t \wedge f)} & = \Box_{f_{1 l_1}} \wedge \dots \wedge \Box_{f_{t' l_{t'}}} \wedge \Diamond_{(e_{t'+1} \wedge f)} \wedge \dots \wedge \Diamond_{(e_t \wedge f)}\\ & \le z \mathrel{\Ks(T)} y
 \le \Diamond_{c_j} \vee \Box_{d_{j1}} \vee \cdots \vee \Box_{d_{j m_j}}.
\end{align*}
Thus, there exists $k \le m_j$ such that $f \T d_{jk}$ or there exists $r > t'$ such that $(e_r \wedge f) \T c_j$. In the first case,
$a_i \S \bigwedge_{r=1}^{t'} f_{r l_r} = f \T d_{jk}$, and hence $a_i \mathrel{(T \circ S)} d_{jk}$. In the second case, there is $s \le n_i$ such that $b_{i s} \S e_r$. Since $x$ is written in conjunctive normal form,
$b_{is} \le a_i \S f$, and hence $b_{is} \S f$. Thus, $b_{is} \S (e_r \wedge f) \T c_j$, which implies that $b_{is} \mathrel{(T \circ S)} c_j$. In either case, we have that $x \mathrel{\Ks(T \circ S)} y$.

Recall from \cref{sec: subordinations} that the identity on $B \in \BAS$ is the partial order $\le$. By \Cref{lem:le in KB}, $\Ks(\le)$ is the partial order $\le$ on $\K(B)$. Therefore, $\Ks$ preserves identities, hence is an endofunctor on $\BAS$.
\end{proof}

The dagger $(-)^\dagger \colon \KHausR \to \KHausR$ introduced before \cref{l:Vr and dagger} clearly restricts to $\StoneR$. The corresponding dagger $(-)^\dagger$ on $\BAS$ is given by mapping a subordination $S \colon A \to B$ to the subordination $S^\dagger \colon B \to A$ defined by $b \S^\dagger a$ iff $\neg a \S \neg b$ (see \cite[Thm.~2.14]{ABC22a}).
Since $(-)^\dagger \colon \BAS(A,B) \to \BAS(B,A)$ preserves inclusion, it gives an order-isomorphism between $\BAS(A,B)$ and $\BAS(B,A)$.
We conclude this section by observing that $\Ksbox$ and $\Ksdia$ are definable from each other using $(-)^\dagger$ and that $\Ks$ commutes with $(-)^\dagger$.

\begin{proposition}\label{l:Ks and dagger}
\hfill\begin{enumerate}[label=\normalfont(\arabic*), ref = \arabic*]
\item\label{l:Ks and dagger:item1} $\Ksbox \circ (-)^\dagger = (-)^\dagger \circ \Ksdia$ and $\Ksdia \circ (-)^\dagger = (-)^\dagger \circ \Ksbox$.
\item\label{l:Ks and dagger:item2} $\Ks \circ (-)^\dagger = (-)^\dagger \circ \Ks$.
\end{enumerate}
\end{proposition}

\begin{proof}
Since $\Ksbox, \Ksdia, \Ks$ coincide on objects and $(-)^\dagger$ fixes the objects, we only need to show that the compositions agree on the morphisms.

\eqref{l:Ks and dagger:item1}. 
We only prove the first equality because the second is proved similarly. Let $S \colon A \to B$ be a morphism in $\BAS$. Suppose that $x \in \K(A)$ and $y \in \K(B)$ are written in disjunctive and conjunctive normal form
\begin{align*}
x=\bigvee_{i=1}^n (\Box_{a_i} \wedge \Diamond_{b_{i1}} \wedge \cdots \wedge \Diamond_{b_{i n_i}}), \quad y=\bigwedge_{j=1}^m (\Diamond_{c_j} \vee \Box_{d_{j1}} \vee \cdots \vee \Box_{d_{j m_j}}).
\end{align*}
As $\neg \Box_a=\Diamond_{\neg a}$ and $\neg \Diamond_b=\Box_{\neg b}$, we obtain that $\neg x$ and $\neg y$ can be written in conjunctive and disjunctive normal form
\begin{align*}
\neg x=\bigwedge_{i=1}^n (\Diamond_{\neg a_i} \vee \Box_{\neg b_{i1}} \vee \cdots \vee \Box_{\neg b_{i n_i}}), \quad \neg y=\bigvee_{j=1}^m (\Box_{\neg c_j} \wedge \Diamond_{\neg d_{j1}} \wedge \cdots \wedge \Diamond_{\neg d_{j m_j}}).
\end{align*}
By \Cref{p:Sbox Sdia independent representatives}, $x \mathrel{\Ksbox(S^\dagger)} y$ iff for all $i,j$ there exists $k$ such that $a_i \mathrel{S^\dagger} d_{jk}$, which means that $\neg d_{jk} \S \neg a_i$. Applying \Cref{p:Sbox Sdia independent representatives} again, this is equivalent to $\neg y \mathrel{\Ksdia(S)} \neg x$, which means that $x \mathrel{\Ksdia(S)^\dagger} y$. 

\eqref{l:Ks and dagger:item2}.
Since $(-)^\dagger \colon \BAS(A,B) \to \BAS(B,A)$ is an order-isomorphism, $(S_1 \vee S_2)^\dagger=S_1^\dagger \vee S_2^\dagger$ for each $S_1,S_2 \in \BAS(A,B)$. Therefore, by \eqref{l:Ks and dagger:item1}, for each subordination $S \colon A \to B$ we have
\[
\Ks(S^\dagger) = \Ksbox(S^\dagger) \vee \Ksdia(S^\dagger) = \Ksdia(S)^\dagger \vee \Ksbox(S)^\dagger = (\Ksdia(S) \vee \Ksbox(S))^\dagger = \Ks(S)^\dagger.\qedhere
\]
\end{proof}

\begin{remark}
By \cref{l:Ks and dagger}\eqref{l:Ks and dagger:item2}, $\Ks$ commutes with the dagger on $\BAS$. It follows from \cref{c:Sem description}  that $\Ks$ preserves inclusions of subordinations. Thus, $\Ks$ is a morphism of order-enriched categories with involution. However, like $\Vr$, the functor $\Ks$ is not a morphism of allegories. This follows from \cref{ex:VR not morph allegories} and Theorems~\ref{t:StoneR BAS equiv}, \ref{t:commute}\eqref{t:commute:item3}.
\end{remark}

\section{Connecting \texorpdfstring{$\Ks$}{K{\textasciicircum}S} and \texorpdfstring{$\Vr$}{V{\textasciicircum}R}} \label{sec: connecting K and V}
In this section we show that, under the equivalence between $\StoneR$ and $\BAS$, the endofunctor $\Vr$ on $\StoneR$ corresponds to the endofunctor $\Ks$ on $\BAS$. Analogous results are obtained for $\Vrbox$ and $\Ks_\Box$ as well as for $\Vrdia$ and $\Ks_\Diamond$.

\begin{proposition} \label{prop:description-on-morphisms}
	Suppose that $R \colon X \to Y$ is a closed relation between Stone spaces. Let $\mathcal{U}\in\Clop(\V(X))$ and $\mathcal{W} \in \Clop(\V(Y))$ be written in disjunctive and conjunctive normal form
	\[
	\mathcal{U} = \bigcup_{i=1}^n (\Box_{U_i} \cap \Diamond_{V_{i 1}} \cap \cdots \cap \Diamond_{V_{i n_i}}), \quad \mathcal{W} = \bigcap_{j=1}^m (\Diamond_{W_j} \cup \Box_{Z_{j 1}} \cup \cdots \cup \Box_{Z_{j m_j}})
	\]
	with clopen sets $\varnothing \neq V_{il} \subseteq U_i$ for each $i,l$, and $W_j \subseteq Z_{jk} \neq Y$ for each $j,k$.
	\begin{enumerate}[label=\normalfont(\arabic*), ref = \arabic*]
		\item \label{i:dual-box}
		$\mathcal{U} \mathrel{S_{\Rbox}} \mathcal{W}$ iff $\forall i,j \, \exists k \le m_j : U_i \mathrel{S_R} Z_{jk}$,
		\item \label{i:dual-diamond}
		$\mathcal{U} \mathrel{S_{\Rdia}} \mathcal{W}$ iff $\forall i,j \, \exists l \le n_i : V_{il} \mathrel{S_R} W_j$,
		\item \label{i:dual-EM}
		$\mathcal{U} \mathrel{S_{\Rem}} \mathcal{W}$ iff $\forall i,j \, ((\exists k \le m_j : U_i \mathrel{S_R} Z_{jk}) \vee (\exists l \le n_i : V_{il} \mathrel{S}_R W_j))$.
	\end{enumerate}
\end{proposition}

\begin{proof}
	\eqref{i:dual-box}.
	Since $S_{\Rbox}$ is a subordination, 
	it is enough to prove the case $n = m = 1$, i.e.\ that
	\[
	(\Box_U\cap\Diamond_{V_1}\cap\cdots\cap\Diamond_{V_p}) \mathrel{S_{\mathcal{R}_\Box}} (\Diamond_W \cup \Box_{Z_1} \cup \cdots \cup \Box_{Z_q})\mbox{ iff }U \S_R Z_k\mbox{ for some }k \le q,
	\]
	under the assumption that	$\varnothing \neq V_i \subseteq U$ and $W \subseteq Z_j \neq Y$.
	
	To prove the right-to-left implication, suppose that $U \S_R Z_k$. Then	$R[U] \subseteq Z_k$.
	Let $F \in \Box_U \cap \Diamond V_1 \cap \dots \cap \Diamond V_p$ and 
	$G \in \V(Y)$ with $F \mathrel{\Rbox} G$, so $G \subseteq R[F]$.
	Since $F \in \Box_U$, we have $F \subseteq U$, and hence $G \subseteq R[F] \subseteq R[U] \subseteq Z_k$. Thus, $G \in \Box_{Z_k} \subseteq \Diamond_W \cup \Box_{Z_1} \cup \cdots \cup \Box_{Z_q}$, and so	$(\Box_U\cap\Diamond_{V_1}\cap\cdots\cap\Diamond_{V_p}) \mathrel{S_{\Rbox}} (\Diamond_W \cup \Box_{Z_1} \cup \cdots \cup \Box_{Z_q})$.
	
	To prove the other implication, 
	suppose that $U \not\S_R Z_k$ for each $k \le q$, so there are $x_k \in U$ and $y_k \in - Z_k$ such that $x_k \R y_k$.
	Then 
\begin{align*}
&\{x_1, \dots, x_q\} \cup V_1 \cup \dots \cup V_p \in \Box_U \cap \Diamond_{V_1}\cap\cdots\cap\Diamond_{V_p},\\
&\{y_1, \dots, y_q\} \notin \Diamond_W \cup \Box_{Z_1} \cup \cdots \cup \Box_{Z_q}, \text{ and}\\
&\{y_1, \dots, y_q\} \subseteq R[\{x_1, \dots, x_q\}] \subseteq R[\{x_1, \dots, x_q\} \cup V_1 \cup \dots \cup V_p].
\end{align*}
	Therefore, $\{x_1, \dots, x_q\} \cup V_1 \cup \dots \cup V_n \ \mathcal{R}_\Box \ \{y_1, \dots, y_q\}$, proving the failure of 
	\[
	(\Box_U\cap\Diamond_{V_1}\cap\cdots\cap\Diamond_{V_p}) \mathrel{S_{\mathcal{R}_\Box}} (\Diamond_W \cup \Box_{Z_1} \cup \cdots \cup \Box_{Z_q}).
	\]

	\eqref{i:dual-diamond}.	
	Since $S_{\mathcal{R}_\Diamond}$ is a subordination, it is enough to prove the case $n = m = 1$, i.e.\ that
	\[
	(\Box_U\cap\Diamond_{V_1}\cap\cdots\cap\Diamond_{V_p}) \mathrel{S_{\mathcal{R}_\Diamond}} (\Diamond_W \cup \Box_{Z_1} \cup \cdots \cup \Box_{Z_q})\mbox{ iff }V_l \S_R W\mbox{ for some }l \le p,
	\]
	under the assumption that $\varnothing \neq V_i \subseteq U$ and $W \subseteq Z_j \neq Y$.
	
	For the right-to-left implication, suppose that $V_l \S_R W$. Then $R[V_l] \subseteq W$.
	Let $F \in	\Box_U\cap\Diamond_{V_1}\cap\cdots\cap\Diamond_{V_p}$ and $G \in \V(Y)$ with $F \mathrel{\mathcal{R}_\Diamond} G$, so $F \subseteq R^{-1}[G]$.
	Since $F \in \Diamond_{V_l}$, we have $F \cap V_l \neq \varnothing$. 
	Therefore, $R^{-1}[G] \cap V_l \neq \varnothing$,	and hence $G \cap R[V_l] \neq \varnothing$. 
	Thus, $G \cap W \neq \varnothing$.
	Consequently, $G \in \Diamond_W$, which implies that $G \in \Diamond_W \cup \Box_{Z_1} \cup \cdots \cup \Box_{Z_q}$.
	This shows that $(\Box_U\cap\Diamond_{V_1}\cap\cdots\cap\Diamond_{V_p}) \mathrel{S_{\mathcal{R}_\Diamond}} (\Diamond_W \cup \Box_{Z_1} \cup \cdots \cup \Box_{Z_q})$.
	
	For the left-to-right implication, suppose that $V_l \not\S_R W$ for each $l \le p$, so there are $x_l \in V_l$ and $y_l \in -W$ such that $x_l \R y_l$.
	Then 
\begin{align*}
&\{x_1, \dots, x_p\} \in \Box_U \cap \Diamond_{V_1}\cap\cdots\cap\Diamond_{V_p},\\
&\{y_1, \dots, y_p\} \cup -Z_1 \cup \dots \cup -Z_q \notin \Diamond_W \cup \Box_{Z_1} \cup \cdots \cup \Box_{Z_q}, \text{ and}\\
&\{x_1, \dots, x_p\} \subseteq R^{-1}[\{y_1, \dots, y_p\}] \subseteq R^{-1}[\{y_1, \dots, y_p\} \cup -Z_1 \cup \dots \cup -Z_q],
\end{align*}
	which implies that $\{x_1, \dots, x_p\} \ \mathcal{R}_\Diamond \ \{y_1, \dots, y_p\} \cup -Z_1 \cup \dots \cup -Z_q$,	proving the failure of
	\[
	(\Box_U\cap\Diamond_{V_1}\cap\cdots\cap\Diamond_{V_p}) \mathrel{S_{\mathcal{R}_\Diamond}} (\Diamond_W \cup \Box_{Z_1} \cup \cdots \cup \Box_{Z_q}).
	\]
	
	\eqref{i:dual-EM}.
	Since $S_{\mathcal{R}_\mathrm{EM}}$ is a subordination, it is enough to prove the case $n = m = 1$, i.e.\ that 
	\[
		(\Box_U\cap\Diamond_{V_1}\cap\cdots\cap\Diamond_{V_p}) \mathrel{S_{\mathcal{R}_\mathrm{EM}}} (\Diamond_W \cup \Box_{Z_1} \cup \cdots \cup \Box_{Z_q})\mbox{ iff } \exists k: U \S_R Z_k\mbox{ or }\exists l : V_l \S_R W,
	\]
	under the assumption that	$\varnothing \neq V_i \subseteq U$ and $W \subseteq Z_j \neq Y$.
	
	For the right-to-left implication,
	suppose that $U \S_R Z_k$ for some $k \le q$ or $V_l \S_R W$ for some $l \le p$.
	Then, by \eqref{i:dual-box} and \eqref{i:dual-diamond}, $(\Box_U\cap\Diamond_{V_1}\cap\cdots\cap\Diamond_{V_p}) \mathrel{S_{\mathcal{R}_\Box}} (\Diamond_W \cup \Box_{Z_1} \cup \cdots \cup \Box_{Z_q})$ or $(\Box_U\cap\Diamond_{V_1}\cap\cdots\cap\Diamond_{V_p}) \mathrel{S_{\mathcal{R}_\Diamond}} (\Diamond_W \cup \Box_{Z_1} \cup \cdots \cup \Box_{Z_q})$. Thus,
	\[
	(\Box_U\cap\Diamond_{V_1}\cap\cdots\cap\Diamond_{V_p}) \mathrel{S_{\mathcal{R}_\mathrm{EM}}} (\Diamond_W \cup \Box_{Z_1} \cup \cdots \cup \Box_{Z_q}).
	\]
	
	For the other implication, suppose that for each $k \le q$ there are $x_k \in U$ and $y_k \in -Z_k$	such that $x_k \R y_k$, and that for each $l \le p$ there are $x'_l \in V_l$ and $y'_l \in -W$	such that $x'_l \R y'_l$.
	Therefore, 
\begin{align*}
&\{x_1, \dots, x_q, x'_1, \dots, x'_p\} \in \Box_U \cap \Diamond_{V_1}\cap\cdots\cap\Diamond_{V_p},\\
&\{y_1, \dots, y_q, y'_1, \dots, y'_p\} \notin \Diamond_W \cup \Box_{Z_1} \cup \cdots \cup \Box_{Z_q},\\
& \{y_1, \dots, y_q, y_1', \dots, y_p'\} \subseteq R[\{x_1, \dots, x_q, x_1', \dots, x_p'\}], \text{ and}\\
& \{x_1, \dots, x_q, x'_1, \dots, x'_p\} \subseteq R^{-1}[\{y_1, \dots, y_q, y'_1, \dots, y'_p\}].
\end{align*}
	Thus, $\{x_1, \dots, x_q, x'_1, \dots, x'_p\} \ \mathcal{R}_\mathrm{EM} \ \{y_1, \dots, y_q, y_1', \dots, y_p'\}$, proving the failure of 
	\[
	(\Box_U\cap\Diamond_{V_1}\cap\cdots\cap\Diamond_{V_p}) \mathrel{S_{\mathcal{R}_\mathrm{EM}}} (\Diamond_W \cup \Box_{Z_1} \cup \cdots \cup \Box_{Z_q}).\qedhere
	\]
\end{proof}

We next recall the notions of a natural transformation and natural isomorphism between semi-functors.

\begin{definition}[{\cite[Sec.~2.2]{Hoofman1993}}]
	Let $\mathsf{D}, \mathsf{E}$ be categories and $F, G, H \colon \mathsf{D} \to \mathsf{E}$ semi-functors.
	\begin{enumerate}
		\item A \emph{natural transformation} (called ``semi natural transformation'' in \cite[Def.~2.4]{Hoofman1993}) $\alpha$ from $F$ to $G$ is a function which assigns to each object $D \in \mathsf{D}$ a morphism $\alpha_D \colon F(D) \to G(D)$ in $\mathsf{E}$ such that for every morphism $f \colon D \to D'$ in $\mathsf{D}$ the following square commutes
		\[
		\begin{tikzcd}[column sep=5pc, row sep=4pc]
			F(D) \arrow{r}{\alpha_D} \arrow[swap]{d}{F(f)} & G(D)\arrow{d}{G(f)}\\
			F(D') \arrow{r}{\alpha_{D'}}& G(D')
		\end{tikzcd}
		\]
		and for every $D \in \mathsf{D}$ the following diagram commutes (under the hypotheses above, the commutativity of one triangle implies the commutativity of the other triangle).
		\[
		\begin{tikzcd}[column sep=5pc, row sep=4pc]
			F(D) \arrow[swap]{d}{F(\mathrm{id}_D)}\arrow{rd}{\alpha_D} \arrow{r}{\alpha_D} & G(D) \arrow{d}{G(\mathrm{id}_D)}\\
			F(D)\arrow[swap]{r}{\alpha_D}  & G(D)
		\end{tikzcd}
		\]
		\item 
		The composition $\beta \circ \alpha \colon F \to H$ of natural transformations $\alpha \colon F \to G$ and $\beta \colon G \to H$ is the natural transformation defined by $(\beta \circ \alpha)_D = \beta_D \circ \alpha_D$ for $D \in \mathsf{D}$.
		
		\item
		The identity natural transformation $1_F \colon F \to F$ on $F$ has components $(1_F)_D \coloneqq F(\mathrm{id_D})$. (Note that, in general, $(1_F)_D \neq \mathrm{id}_{F(D)}$.)
		
		\item $F$ and $G$ are \emph{naturally isomorphic} (denoted $F \simeq G$) if there are natural transformations $\alpha \colon F \to G$ and $\beta \colon G \to F$ such that $\alpha \circ \beta = 1_G$ and $\beta \circ \alpha = 1_F$.
	\end{enumerate}
\end{definition}

\begin{theorem}\label{t:commute}
We have the following natural isomorphisms:
\begin{enumerate}[label=\normalfont(\arabic*), ref = \arabic*]
\item\label{t:commute:item1}  $\Clop\circ\Vrbox \simeq \Ksbox\circ\Clop$ and $\Uf\circ\Ksbox \simeq \Vrbox\circ\Uf$.
\item\label{t:commute:item2}  $\Clop\circ\Vrdia \simeq \Ksdia\circ\Clop$ and $\Uf\circ\Ksdia \simeq \Vrdia\circ\Uf$.
\item\label{t:commute:item3}  $\Clop\circ\Vr \simeq \Ks\circ\Clop$ and $\Uf\circ\Ks \simeq \Vr\circ\Uf$.
\end{enumerate}
Therefore, the following three diagrams commute up to natural isomorphism.\\
\noindent
\begin{minipage}[b]{0.33333\textwidth}%
\[
	\begin{tikzcd}[column sep=5pc, row sep=4pc]
		\StoneR \arrow[r, shift left=1, "\Clop"] \arrow[d, "\Vrbox"'] & \BAS \arrow[l, shift left=1, "\Uf"]  \arrow[d, "\Ksbox"] \\
		\StoneR \arrow[r, shift left=1, "\Clop"]  & \BAS \arrow[l, shift left=1, "\Uf"]
	\end{tikzcd}
	\]
\end{minipage}%
\begin{minipage}[b]{0.33333\textwidth}
\[
	\begin{tikzcd}[column sep=5pc, row sep=4pc]
		\StoneR \arrow[r, shift left=1, "\Clop"] \arrow[d, "\Vrdia"'] & \BAS \arrow[l, shift left=1, "\Uf"]  \arrow[d, "\Ksdia"] \\
		\StoneR \arrow[r, shift left=1, "\Clop"]  & \BAS \arrow[l, shift left=1, "\Uf"]
	\end{tikzcd}
	\]
\end{minipage}%
\begin{minipage}[b]{0.33333\textwidth}
\[
	\begin{tikzcd}[column sep=5pc, row sep=4pc]
		\StoneR \arrow[r, shift left=1, "\Clop"] \arrow[d, "\Vr"'] & \BAS \arrow[l, shift left=1, "\Uf"]  \arrow[d, "\Ks"] \\
		\StoneR \arrow[r, shift left=1, "\Clop"]  & \BAS \arrow[l, shift left=1, "\Uf"]
	\end{tikzcd}
	\]
\end{minipage}
\end{theorem}

\begin{proof}
	\eqref{t:commute:item1}.
	We only prove that $\Clop\circ\Vrbox \simeq \Ksbox\circ\Clop$; that $\Uf\circ\Ksbox \simeq \Vrbox\circ\Uf$ follows since $\Clop$ and $\Uf$ are quasi-inverses.
	By \cite[Fact~1]{VV14}, for each $X \in \Stone$, there is	a boolean isomorphism $f_X \colon \Clop(\V(X)) \to \K(\Clop(X))$, which maps, for every clopen $U$ of $X$, the clopen $\Box_U$ of $\V(X)$ to the formal symbol $\Box_U \in \K(\Clop(X))$, and the clopen $\Diamond_U$ 
	of $\V(X)$ to the formal symbol	$\Diamond_U \in \K(\Clop(X))$.\footnote{We have a conflict of notation here since, on the one hand, $\Box_U$ denotes a clopen subset of $\V(X)$, and on the other hand, a formal symbol in $\K(\Clop(X))$. This will come up again in \cref{claim}, but nowhere else.}
	For each $X \in \StoneR$, we define two relations $\alpha_X \colon \Clop(\Vrbox(X)) \to \Ksbox(\Clop(X))$ and $\beta_X \colon \Ksbox(\Clop(X)) \to \Clop(\Vrbox(X))$ by setting 
	\[
	\mathcal{U} \mathrel{\alpha_X} f_X(\mathcal{V}) \iff f_X(\mathcal{U}) \mathrel{\beta_X} \mathcal{V} \iff \mathcal{U} \mathrel{\Clop(\Vrbox(\id_X))} \mathcal{V}
	\]
	for each $\mathcal{U}, \mathcal{V} \in \Clop(\Vrbox(X))$. 
	For the sake of completeness, we give its explicit formulation:
	\begin{align*}
		\mathcal{U} \mathrel{\alpha_X} f_X(\mathcal{V})\ &\iff \forall F \in \mathcal{U}\ \forall G \in \V(X)\, (G \subseteq F \Rightarrow G \in \mathcal{V}),\\
		f_X(\mathcal{U}) \mathrel{\beta_X} \mathcal{V}\ &\iff \forall F \in \mathcal{U}\ \forall G \in \V(X)\, (G \subseteq F \Rightarrow G \in \mathcal{V}).
	\end{align*}
	We prove that $\alpha$ and $\beta$ are natural transformations.
	Since $\Clop(\Vrbox(1_X))$ is a subordination and $f_X$ a boolean isomorphism, $\alpha_X$ and $\beta_X$ are subordinations.
	Let $R \colon X \to Y$ be a morphism in $\StoneR$.
	We prove that the following square commutes.
	\[
		\begin{tikzcd}[column sep=3pc, row sep=4pc]
			\Clop(\Vrbox(X)) \arrow{r}{\alpha_X} \arrow[swap]{d}{\Clop(\Vrbox(R))} & \Ksbox(\Clop(X))\arrow{d}{\Ksbox(\Clop(R))}\\
			\Clop(\Vrbox(Y)) \arrow{r}{\alpha_{Y}}& \Ksbox(\Clop(Y))
		\end{tikzcd}
	\]
	For this we require the following claim.
	\begin{claim}\label{claim}
	For $\mathcal{U} \in \Clop(\Vrbox(X))$ and $\mathcal{W} \in \Clop(\Vrbox(Y))$ we have
	\[
	\mathcal{U} \mathrel{\Clop(\Vrbox(R))} \mathcal{W} \iff f_X(\mathcal{U}) \mathrel{\Ksbox(\Clop(R))} f_Y(\mathcal{W}).
	\]
	\end{claim}
	
	\begin{proof}[Proof of claim]
	Write $\mathcal{U}$ and $\mathcal{W}$
	in disjunctive and conjunctive normal form
		\[
		\mathcal{U} = \bigcup_{i=1}^n (\Box_{U_i} \cap \Diamond_{V_{i 1}} \cap \cdots \cap \Diamond_{V_{i n_i}}), \quad \mathcal{W} = \bigcap_{j=1}^m (\Diamond_{W_j} \cup \Box_{Z_{j 1}} \cup \cdots \cup \Box_{Z_{j m_j}})
		\]
		with clopen sets $\varnothing \neq V_{il} \subseteq U_i$ for each $i,l$, and $W_j \subseteq Z_{jk} \neq Y$ for each $j,k$. Since $\Clop(\Vrbox(R)) = S_{\Rbox}$, \cref{prop:description-on-morphisms}\eqref{i:dual-box} yields 
		\[
		\mathcal{U} \mathrel{\Clop(\Vrbox(R))} \mathcal{W} \iff \forall i,j \, \exists k \le m_j : U_i \mathrel{S_R} Z_{jk}.
		\]
		On the other hand, since $f_X$ is a boolean homomorphism,
		\[
		f_X(\mathcal{U}) = \bigvee_{i=1}^n (\Box_{U_i} \wedge \Diamond_{V_{i 1}} \wedge \cdots \wedge \Diamond_{V_{i n_i}}) \quad \text{and} \quad  f_X(\mathcal{W}) = \bigwedge_{j=1}^m (\Diamond_{W_j} \vee \Box_{Z_{j 1}} \vee \cdots \vee \Box_{Z_{j m_j}}).
		\]
		Since $\Ksbox(\Clop(R)) = \Ksbox(S_R)$, \cref{p:Sbox Sdia independent representatives}\eqref{p:Sbox Sdia independent representatives:item1} yields
		\[
		f_X(\mathcal{U}) \mathrel{\Ksbox(\Clop(R))} f_X(\mathcal{W}) \iff \forall i,j \, \exists k \le m_j : U_i \mathrel{S_R} Z_{jk}.
		\]
		Thus, $\mathcal{U} \mathrel{\Clop(\Vrbox(R))} \mathcal{W}$ iff $f_X(\mathcal{U}) \mathrel{\Ksbox(\Clop(R))} f_Y(\mathcal{W})$.
	\end{proof}
	To prove the commutativity of the square, let $\mathcal{U} \in \Clop(\Vrbox(X))$ and $\mathcal{W} \in \Clop(\Vrbox(Y))$.
	On the one hand,
	\begin{align*}
		& \mathcal{U} \mathrel{(\alpha_Y \circ \Clop(\Vrbox(R)))} f_Y(\mathcal{W}) \\
		& \iff \exists \mathcal{Z} \in \Clop(\Vrbox(Y)) : \mathcal{U} \mathrel{\Clop(\Vrbox(R))} \mathcal{Z} \mathrel{\alpha_Y} f_Y(\mathcal{W})\\
		& \iff \exists \mathcal{Z} \in \Clop(\Vrbox(Y)) : \mathcal{U} \mathrel{\Clop(\Vrbox(R))} \mathcal{Z} \mathrel{\Clop(\Vrbox(\id_X))} \mathcal{W}\\
		& \iff \mathcal{U} \mathrel{(\Clop(\Vrbox(\id_X)) \circ \Clop(\Vrbox(R)))} \mathcal{W}\\
		& \iff \mathcal{U} \mathrel{\Clop(\Vrbox(\id_X \circ R))} \mathcal{W}\\
		& \iff \mathcal{U} \mathrel{\Clop(\Vrbox(R))} \mathcal{W}.
	\end{align*}
	On the other hand,
	\begin{align*}
		&\mathcal{U} \mathrel{(\Ksbox(\Clop(R)) \circ \alpha_X)} f_Y(\mathcal{W})\\
		& \iff \exists \mathcal{V} \in \Clop(\Vrbox(X)) : \mathcal{U} \mathrel{\alpha_X} f_X(\mathcal{V}) \mathrel{\Ksbox(\Clop(R))} f_Y(\mathcal{W})\\
		& \iff \exists \mathcal{V} \in \Clop(\Vrbox(X)) : \mathcal{U} \mathrel{\alpha_X} f_X(\mathcal{V}),\, \mathcal{V} \mathrel{\Clop(\Vrbox(R))} \mathcal{W} & \text{by \cref{claim}}\\
		& \iff \exists \mathcal{V} \in \Clop(\Vrbox(X)) : \mathcal{U} \mathrel{\Clop(\Vrbox(\id_X))} \mathcal{V} 
		\mathrel{\Clop(\Vrbox(R))} \mathcal{W}\\
		& \iff \mathcal{U} \mathrel{(\Clop(\Vrbox(R)) \circ \Clop(\Vrbox(\id_X)))} \mathcal{W}\\
		& \iff \mathcal{U} \mathrel{\Clop(\Vrbox(R \circ \id_X))} \mathcal{W}\\		
		& \iff \mathcal{U} \mathrel{\Clop(\Vrbox(R))} \mathcal{W}.
	\end{align*}
	Thus, $\alpha_Y \circ \Clop(\Vrbox(R)) = \Ksbox(\Clop(R)) \circ \alpha_X$.
	Moreover, when $X = Y$ and $R = \id_X$, the calculation in the second to last display above and the definition of $\alpha_X$ give
	\begin{align*}
		\mathcal{U} \mathrel{(\alpha_X \circ \Clop(\Vrbox(\id_X)))} f_X(\mathcal{V}) & \iff \mathcal{U} \mathrel{\Clop(\Vrbox(\id_X))} \mathcal{V} \iff \mathcal{U} \mathrel{\alpha_X} f_X(\mathcal{V}).
	\end{align*}
	Therefore, the following triangle commutes for every $X \in \StoneR$.
	\[
		\begin{tikzcd}[column sep=5pc, row sep=4pc]
			\Clop(\Vrbox(X)) \arrow[swap]{d}{\Clop(\Vrbox(\id_X))}\arrow{rd}{\alpha_X}\\
			\Clop(\Vrbox(X))\arrow[swap]{r}{\alpha_X}  & \Ksbox(\Clop(X))
		\end{tikzcd}
	\]
	Thus, $\alpha$ is a natural transformation.
	That $\beta$ is a natural transformation is proved similarly.
	
	Let $X \in \StoneR$.
	We prove that $\beta_X \circ \alpha_X = \Clop(\Vrbox(\id_X))$.
	For $\mathcal{U}, \mathcal{V} \in \Clop(\Vrbox(X))$ we have
	\begin{align*}
		\mathcal{U} \mathrel{(\beta_X \circ \alpha_X)} \mathcal{V} & \iff \exists \mathcal{W} \in \Clop(\Vrbox(X)) : \mathcal{U} \mathrel{\alpha_X} f_{X}(\mathcal{W}) \mathrel{\beta_X} \mathcal{V}\\
		& \iff \exists \mathcal{W} \in \Clop(\Vrbox(X)) : \mathcal{U} \mathrel{\Clop(\Vrbox(\id_X))} \mathcal{W} \mathrel{\Clop(\Vrbox(\id_X))} \mathcal{V}\\
		& \iff \mathcal{U} \mathrel{\Clop(\Vrbox(\id_X)) \circ \Clop(\Vrbox(\id_X))} \mathcal{V}\\
		& \iff \mathcal{U} \mathrel{\Clop(\Vrbox(\id_X \circ \id_X))} \mathcal{V}\\
		& \iff \mathcal{U} \mathrel{\Clop(\Vrbox(\id_X))} \mathcal{V}.
	\end{align*}
	This proves that $\beta_X \circ \alpha_X = \Clop(\Vrbox(\id_X))$.
	We next prove that $\alpha_X \circ \beta_X = \Ksbox(\Clop(\id_X))$.
	For $\mathcal{U}, \mathcal{V} \in \Clop(\Vrbox(X))$ we have
	\begin{align*}
		& f_X(\mathcal{U}) \mathrel{(\alpha_X \circ \beta_X)} f_X(\mathcal{V}) \\
		& \iff \exists \mathcal{W} \in \Clop(\Vrbox(X)) : f_X(\mathcal{U}) \mathrel{\beta_X} \mathcal{W} \mathrel{\alpha_X} f_X(\mathcal{V})\\
		& \iff \exists \mathcal{W} \in \Clop(\Vrbox(X)) : \mathcal{U} \mathrel{\Clop(\Vrbox(\id_X))} \mathcal{W} \mathrel{\Clop(\Vrbox(\id_X))} f_X(\mathcal{V})\\
		& \iff \mathcal{U} \mathrel{\Clop(\Vrbox(\id_X \circ \id_X))} \mathcal{V}\\
		& \iff \mathcal{U} \mathrel{\Clop(\Vrbox(\id_X))} \mathcal{V}\\
		& \iff f_X(\mathcal{U}) \mathrel{\Ksbox(\Clop(\id_X))} f_X(\mathcal{V}) & \text{by \cref{claim}.}
	\end{align*}
	Therefore, $\alpha_X \circ \beta_X = \Ksbox(\Clop(\id_X))$. Thus,	$\alpha$ and $\beta$ yield a natural isomorphism between $\Clop\circ\Vrbox$ and $\Ksbox\circ\Clop$.

	\eqref{t:commute:item2} and \eqref{t:commute:item3} are proved similarly.
	For the sake of completeness, we describe the natural isomorphisms.
	For \eqref{t:commute:item2}, the natural isomorphism is given by the natural transformations $\alpha \colon \Clop \circ \Vrdia \to \Ksdia \circ \Clop$ and $\beta \colon \Ksdia \circ \Clop \to \Clop \circ \Vrdia$ defined by 
	\begin{align*}
	\mathcal{U} \mathrel{\alpha_X} f_X(\mathcal{V}) &\iff f_X(\mathcal{U}) \mathrel{\beta_X} \mathcal{V} \iff \mathcal{U} \mathrel{\Clop(\Vrdia(\id_X))} \mathcal{V}\\
	&\iff \forall F \in \mathcal{U}\ \forall G \in \V(X)\, (G \subseteq F \Rightarrow G \in \mathcal{V}).
	\end{align*}
	For \eqref{t:commute:item3}, the natural transformations $\alpha \colon \Clop \circ \Vr \to \Ks \circ \Clop$ and $\beta \colon \Ks \circ \Clop \to \Clop \circ \Vr$ witnessing the natural isomorphism are defined by 
	\[
		\mathcal{U} \mathrel{\alpha_X} f_X(\mathcal{V}) \iff f_X(\mathcal{U}) \mathrel{\beta_X} \mathcal{V} \iff \mathcal{U} \mathrel{\Clop(\Vr(\id_X))} \mathcal{V} \iff \mathcal{U} \subseteq \mathcal{V}.\qedhere
	\]
\end{proof}

\section{Extending \texorpdfstring{$\Ks$}{K{\textasciicircum}S} to \texorpdfstring{$\SubSfive$}{SubS5{\textasciicircum}S}, \texorpdfstring{$\devS$}{DeV{\textasciicircum}S}, and \texorpdfstring{$\KRFrmP$}{KRFrm{\textasciicircum}P}} \label{sec: extending K to L}

In this section we lift the endofunctor $\Ks$ on $\BAS$ to an endofunctor on $\SubSfive$. We then use MacNeille completions of $\sf S5$-subordination algebras developed in \cite{ABC22b} to obtain an endofunctor $\Ls$ on $\devS$, and prove that $\Ls$ is dual to the Vietoris endofunctor $\Vr$ on $\KHausR$. We conclude by utilizing ideal completions of $\sf S5$-subordination algebras to obtain an endofunctor on $\KRFrmP$, thus generalizing Johnstone's construction of the Vietoris endofunctor on $\KRFrm$. 

We recall (see \cite[Rem.~3.11]{ABC22a}) that the equivalence between $\SubSfive$ and $\StoneER$ is obtained by lifting the functors $\Uf$ and $\Clop$ establishing the equivalence between $\BAS$ and $\StoneR$. We next lift $\Vr \colon \StoneR \to \StoneR$ to an endofunctor on $\StoneER$ and $\Ks \colon \BAS \to \BAS$ to an endofunctor on $\SubSfive$. We slightly abuse notation and denote the lift of a functor by the same letter.

\begin{definition}\label{def:lifts of Vr and Ks}
\hfill
\begin{enumerate}
\item Define $\Vr \colon \StoneER \to \StoneER$ by sending $(X,E)$ to $(\Vr(X),\Vr(E))$ and a morphism $R \colon (X_1,E_1) \to (X_2,E_2)$ to $\Vr(R)$.
\item Define $\Ks \colon \SubSfive \to \SubSfive$ by sending $(B,S)$ to $(\Ks(B),\Ks(S))$ and a morphism $T \colon (B_1,S_1) \to (B_2,S_2)$ to $\Ks(T)$.
\end{enumerate}
\end{definition}

To show that $\Vr$ and $\Ks$ are well defined, we require the following two lemmas.

\begin{lemma} \label{lem: VR(E)}
\hfill\begin{enumerate}[label=\normalfont(\arabic*), ref = \arabic*]
\item\label{lem: VR(E):item1} If $(X,E) \in \StoneER$, then $\Vr(E)$ is a closed equivalence relation on $\Vr(X)$ and $\Vr(X)/\Vr(E) \cong \Vr(X/E)$.
\item\label{lem: VR(E):item2} If $R \colon (X_1,E_1) \to (X_2,E_2)$ is a morphism in $\StoneER$, then $\Vr(\Q(R))=\Q(\Vr(R))$.
\end{enumerate}
\end{lemma}

\begin{proof}
\eqref{lem: VR(E):item1}. 
Since $E=E^\dagger$, it follows from \cref{lem:R R1 R2 closed} that $\Vr(E)$ is a closed equivalence relation. Therefore, if $\pi\colon X\to X/E$ is the quotient map, then for $F,G \in \Vr(X)$ we have 
\[
F \mathrel{\Vr(E)} G \iff E[F]=E[G] \iff \pi^{-1}\pi[F]=\pi^{-1}\pi[G] \iff \pi[F]=\pi[G].
\]
Thus, $\Vr(E)$ is the kernel of $\V(\pi)$, and hence
$\Vr(X)/\Vr(E)$ is homeomorphic to $\Vr(X/E)$.

\eqref{lem: VR(E):item2}.
Since $\Vr$ is a functor and commutes with the dagger, we have
\[
\Q(\Vr(R))=\Vr(\pi) \circ \Vr(R) \circ \Vr(\pi)^\dagger=\Vr(\pi \circ R \circ \pi^\dagger)=\Vr(\Q(R)). \qedhere
\]
\end{proof}

\begin{lemma}\label{th:K-preserves-morphisms}
	Let $S$ be an $\mathsf{S5}$-subordination on a boolean algebra $B$.
	\begin{enumerate}[label=\normalfont(\arabic*), ref = \arabic*]
		\item \label{i:K-S4}
		$\Sbox$ and $\Sdia$ satisfy \emph{\eqref{S5}} and \emph{\eqref{S7}}.
		\item \label{i:K-S5}
		$\mathcal{S}$ is an $\mathsf{S5}$-subordination.
	\end{enumerate}
\end{lemma}

\begin{proof}
\eqref{i:K-S4}.
That $\Sbox$ and $\Sdia$ satisfy \eqref{S5} is an immediate consequence of \cref{def:proximity on KB-12,lem:le in KB}. 
We prove that $\Sbox$ satisfies \eqref{S7}.
Since $\Ksbox$ preserves inclusion of subordinations and composition, $\Sbox \circ \Sbox = \Ksbox(S) \circ \Ksbox(S) = \Ksbox(S \circ S) \subseteq \Ksbox(S) = \Sbox$. Thus, $\Sbox$ satisfies \eqref{S7}. That $\Sdia$ satisfies \eqref{S7} is proved similarly.

\eqref{i:K-S5}.
By \eqref{i:K-S4}, to show that $\Sem$ satisfies \eqref{S5} and \eqref{S7}, it is sufficient to prove that if $S,T \in \BAS(B,B)$ satisfy \eqref{S5} and \eqref{S7}, then so does their join. 
Since \eqref{S5} holds for $S,T$, we have that $S$ and $T$ are contained in $\le$, where $\le$ is the partial order on $B$. Thus, $S \vee T$ is contained in $\le$, and so $S \vee T$ satisfies \eqref{S5}. Since $S,T$ satisfy \eqref{S7}, $S \subseteq S \circ S$ and $T \subseteq T \circ T$. Thus, 
\[
S \vee T \subseteq (S \circ S) \vee (T \circ T) \subseteq (S \vee T) \circ (S \vee T),
\]
where the last inclusion follows from the fact that both $S \circ S$ and $T \circ T$ are contained in $(S \vee T) \circ (S \vee T)$ because $S,T \subseteq S \vee T$. Therefore, $S \vee T$ satisfies \eqref{S7}.

It remains to show that $\mathcal{S}$ satisfies \eqref{S6}.
We have that $S^\dagger=S$ because $S$ satisfies \eqref{S6}. By \Cref{l:Ks and dagger}, $\Sem^\dagger=\Ks(S)^\dagger=\Ks(S^\dagger)=\Ks(S)=\Sem$. Thus, $\Sem$ satisfies \eqref{S6}.
\end{proof}

\begin{theorem} \label{thm: L and K(Vr)}
$\Vr \colon \StoneER \to \StoneER$ and $\Ks \colon \SubSfive \to \SubSfive$ are well defined functors and the following diagram commutes up to natural isomorphism.
\[
\begin{tikzcd}[column sep=5pc, row sep=4pc]
	\KHausR \arrow[d, swap, "\Vr"] \arrow[r, shift left=1, "\G"] & \StoneER \arrow[r, shift left=1, "\Clop"] \arrow[d, "\Vr"] \arrow[l, shift left=1, "\Q"] & \SubSfive \arrow[l, shift left=1, "\Uf"] \arrow[d, "\Ks"]\\
	\KHausR \arrow[r, shift left=1, "\G"] & \StoneER \arrow[r, shift left=1, "\Clop"] \arrow[l, shift left=1, "\Q"] & \SubSfive \arrow[l, shift left=1, "\Uf"]
\end{tikzcd}
\]
\end{theorem}

\begin{proof}
That $\Vr$ is well defined follows from \cref{lem: VR(E)}\eqref{lem: VR(E):item1}, and that $\Ks$ is well defined from \cref{th:K-preserves-morphisms}\eqref{i:K-S5}.
By \cref{lem: VR(E)}, $\Vr \circ \Q = \Q \circ \Vr$, and since $\G$ is a quasi-inverse of $\Q$ (see \cref{rem:G and Q}), the left square commutes.
Finally, because $\Vr$ and $\Ks$ are defined componentwise (see \cref{def:lifts of Vr and Ks}), \cref{t:commute}\eqref{t:commute:item3} yields commutativity of the right square.
\end{proof}

To obtain an endofunctor on $\devS$, we recall the definition of the MacNeille completion of an $\sf S5$-subordination algebra. For a subset $E$ of a poset, we write $E^u$ for the set of upper bounds and $E^\ell$ for the set of lower bounds of $E$.

\begin{definition}
Let $(B,S)$ be an $\sf S5$-subordination algebra. 
\begin{enumerate}
\item\cite[Def.~3.1]{ABC22b} An ideal $I$ of $B$ is an {\em S-ideal} (\emph{subordination ideal}) if for each $a\in I$ there is $b\in I$ such that $a \mathrel{S} b$. 
\item\cite[Thm.~4.8]{ABC22b} An S-ideal $I$ is {\em normal} if $I = S^{-1} [(S[I^u])^\ell]$.
\end{enumerate}
\end{definition}

For $(B,S)\in\SubSfive$, let $\mathcal{NI}(B,S)$ be the set of normal S-ideals of $(B,S)$ ordered by inclusion. Following~\cite[Def.~4.6]{ABC22b}, we call $\mathcal{NI}(B,S)$ the {\em MacNeille completion} of $(B,S)$.

\begin{theorem} \cite[Prop.~4.4, Thm.~4.12]{ABC22b}
If $(B,S)$ is an $\sf S5$-subordination algebra, then $\mathcal{NI}(B,S)$ is a de Vries algebra. Moreover, the assignment $(B,S) \mapsto \mathcal{NI}(B,S)$ extends to a functor $\M\colon\SubSfive\to\devS$, which is an equivalence whose quasi-inverse is the inclusion $\Delta \colon \devS \to \SubSfive$.
\end{theorem}

\begin{definition}
Define $\Ls \colon \devS \to \devS$ to be the composition $\M\circ\Ks\circ\Delta$.
\end{definition}

Clearly $\Ls$ is a well-defined endofunctor on $\devS$. Moreover, since $\M\colon\SubSfive\to\devS$ is an equivalence whose quasi-inverse is $\Delta \colon \devS \to \SubSfive$, we have:

\begin{theorem} \label{thm: H and L}
The following diagram commutes up to a natural isomorphism.
\[
\begin{tikzcd}[column sep=5pc, row sep=4pc]
	\SubSfive \arrow[r, "\M"] \arrow[d, "\Ks", swap] & \devS \arrow[d,"\Ls"] \\
	\SubSfive \arrow[r, "\M"'] & \devS
\end{tikzcd}
\]
\end{theorem}

We recall from \cite[Def.~4.10]{ABC22a} that the functor $\D \colon \KHausR \to \devS$ sends each compact Hausdorff space $X$ to the de Vries algebra $\RO(X)$ of its regular opens and a closed relation $R \colon X_1 \to X_2$ to the compatible subordination $\D(R) \colon \RO(X_1) \to \RO(X_2)$ defined~by 
\[
U \mathrel{\D(R)} V \iff R[\cl(U)] \subseteq V.
\]
By \cite[Thm.~5.11]{ABC22b}, $\M \circ \Clop$ is naturally isomorphic to $\D \circ \Q$. Therefore, combining \cref{thm: L and K(Vr),thm: H and L}, we obtain:

\begin{theorem} \label{thm: main result}
The following diagram commutes up to natural isomorphism.
\[
	\begin{tikzcd}[column sep=5pc, row sep=4pc]
		\KHausR \arrow[d, "\Vr"] \arrow[rrr, "\D", bend left = 1.3em] & \StoneER \arrow[r, "\Clop"] \arrow[d, "\Vr"] \arrow[l,"\Q", swap] & \SubSfive \arrow[r, "\M"] \arrow[d, "\Ks"] & \devS \arrow[d,"\Ls"] \\
		\KHausR \arrow[rrr, "\D", bend right = 1.3 em] & \StoneER \arrow[r, swap, "\Clop"'] \arrow[l, swap,"\Q"] & \SubSfive \arrow[r, swap, "\M"'] & \devS
	\end{tikzcd}
\]
\end{theorem}

\begin{remark}\label{rem:open problem}
We thus obtain that the Vietoris endofunctor $\Vr$ on $\KHausR$ can be dually described either as the endofunctor $\Ks$ on $\SubSfive$ or as the endofunctor $\Ls$ on $\devS$. The usefulness of the latter description stems from the fact that, unlike in $\SubSfive$, isomorphisms in $\devS$ are given by structure-preserving bijections \cite[Thm.~5.4]{ABC22a}.

In \cite[p.~375]{BBH15b} it was left as an open problem to give a direct description of an endofunctor on $\dev$ that is dual to the Vietoris endofunctor $\V$ on $\KHaus$. Our result above solves a similar problem by showing that it is the endofunctor $\Ls$ that is dual to the Vietoris endofunctor $\Vr$ on $\KHausR$. In~\cite{ABC23} we show how to modify our result to obtain a solution of \cite[p.~375]{BBH15b}.
\end{remark}

We conclude the paper by developing the dual endofunctor $\Vr$ for $\KRFrmP$. For this we recall from \cite[Thm.~3.11]{ABC22b} that associating with each $\sf S5$-subordination algebra $(B,S)$ the frame $\mathcal{SI}(B,S)$ of S-ideals of $(B,S)$ defines a contravariant functor $\mathbb I \colon \SubSfive\to\KRFrmP$. To describe its quasi-inverse, we recall that for a frame $L$, the {\em pseudocomplement} of $a\in L$ is given by $a^* = \bigvee \{ x \in L \mid a \wedge x = 0 \}$. The {\em well-inside relation} on $L$ is then defined by $a \prec b$ iff $a^* \vee b = 1$. The \emph{booleanization} of $L$ is the boolean frame $\mathfrak B L = \{ a\in L \mid a=a^{**} \}$ (see, e.g., \cite{BP96}), and restricting $\prec$ to $\mathfrak B L$ yields a de Vries algebra $(\mathfrak B L,\prec)$ \cite{Bez12}. This defines a contravariant functor $\mathfrak{B} \colon \KRFrmP \to \devS$ \cite[Prop.~4.2]{ABC22b}, and we have:

\begin{theorem} \cite[Thm.~4.16(1)]{ABC22b}
The functors $\mathbb I$ and $\Delta \circ \mathfrak{B}$ form a dual equivalence between $\SubSfive$ and $\KRFrmP$.
\end{theorem}

\begin{definition}
Define $\Jp \colon \KRFrmP \to \KRFrmP$ to be the composition $\mathbb I\circ\Ks\circ\Delta\circ\mathfrak{B}$.
\end{definition}

Clearly $\Jp$ is a well-defined endofunctor on $\KRFrmP$. Moreover, since $\mathbb{I}\colon\SubSfive\to\KRFrmP$ is an equivalence whose quasi-inverse is $\Delta \circ \mathfrak{B} \colon \KRFrmP \to \SubSfive$, we obtain:

\begin{theorem}\label{thm:diagram JI}
The following diagram commutes up to natural isomorphism.
\[
\begin{tikzcd}[column sep=5pc, row sep=4pc]
	\SubSfive \arrow[r, "\mathbb I"] \arrow[d, "\Ks", swap] & \KRFrmP \arrow[d,"\Jp"] \\
	\SubSfive \arrow[r, swap, "\mathbb I"'] & \KRFrmP  
\end{tikzcd}
\]
\end{theorem}

For a compact Hausdorff space $X$, let $\mathcal{O}(X)$ be the compact regular frame of opens of $X$. For a closed relation $R \colon X \to Y$, let $\Box_R\colon\mathcal{O}(Y)\to\mathcal{O}(X)$ be given by $\Box_R U = - R^{-1}-U$. This defines a contravariant functor $\mathcal{O}\colon\KHausR\to\KRFrmP$, which is a dual equivalence \cite{Tow96,JKM01}.
By \cite[Thm.~5.7]{ABC22b}, $\mathcal{O} \circ \Q$ is naturally isomorphic to $\I \circ \Clop$. Thus, putting \cref{thm: L and K(Vr),thm:diagram JI} together yields:

\begin{theorem}
The following diagram commutes up to natural isomorphism.
\[
\begin{tikzcd}[column sep=5pc, row sep=4pc]
	\KHausR \arrow[d, "\Vr"] \arrow[rrr, "\mathcal{O}", bend left = 1.3em] & \StoneER \arrow[r, "\Clop"] \arrow[d, "\Vr"] \arrow[l,"\Q", swap] & \SubSfive \arrow[r, "\mathbb I"] \arrow[d, "\Ks"] & \KRFrmP \arrow[d,"\Jp"]\\
	\KHausR \arrow[rrr, "\mathcal{O}", swap, bend right = 1.3em] & \StoneER \arrow[r, swap, "\Clop"'] \arrow[l, swap,"\Q"] & \SubSfive \arrow[r, "\mathbb I"] & \KRFrmP
\end{tikzcd}
\]
\end{theorem}

\begin{remark}\label{rem:Johnstone}
We thus obtain an endofunctor on $\KRFrmP$ that is dual to $\Vr$. This extends Johnstone's endofunctor $\J$ on $\KRFrm$ that is dual to the Vietoris endofunctor $\V$ on $\KHaus$.
Alternatively, one could generalize Johnstone's construction directly, but the difficulty would be in lifting the morphisms. Indeed, Johnstone writes each element of $\J(L)$ as a join of finite meets of the generators $\{ \Box_a,\Diamond_a \mid a \in L\}$. Since every frame homomorphism $f \colon L \to M$ preserves arbitrary joins and finite meets, it lifts to a frame homomorphism ${\J(f) \colon \J(L) \to \J(M)}$. 
On the other hand, preframe homomorphisms only preserve directed joins, so the above lift could be modified as follows. Write each element as a directed join of finite meets of finite joins of the generators $\{ \Box_a,\Diamond_a \mid a \in L\}$. Then the hard part is in proving that the natural lift of $f \colon L \to M$ to $\Jp(f) \colon \J(L) \to \J(M)$ is well defined.
This issue disappears in our alternate construction of $\Jp$ as the composite $\mathbb I\circ\Ks\circ\Delta\circ\mathfrak{B}$.
\end{remark}

\newcommand{\etalchar}[1]{$^{#1}$}

\end{document}